\documentclass[reqno, 9pt]{amsart}
\usepackage{amsmath,amsxtra,amssymb,latexsym, amscd,amsthm}
\usepackage[unicode]{hyperref}
\usepackage{array,tabularx,longtable,multicol,indentfirst,fancybox,color}
\usepackage{graphicx}
\usepackage{multicol}
\usepackage{mathrsfs}
\usepackage{enumerate}
\advance\hoffset-1truecm\relax

\newtheorem{proposition}{Proposition}[section]
\newtheorem{theorem}{Theorem}[section]
\newtheorem{definition}{Definition}
\newtheorem{lemma}{Lemma}[section]
\newtheorem{corollary}{Corollary}[section]
\newtheorem{remark}{Remark}
\newtheorem{example}{Example}[section]
\numberwithin{equation}{section}

\def\R{\mathbb{R}}

\def\vp{\varphi}

\def\lv{\lVert}

\numberwithin{equation}{section}

\begin{document}
	\title[Boundedness in $L_p$ spaces for  Hartley-Fourier convolutions \&  applications]{Boundedness in $L_p$ spaces for the Hartley-Fourier \\convolutions operator and their applications}
	\thanks{Accepted 02 May 2023 by \textit{J. Math. Sci.}}
	\thanks{Published
		11 July 2023}
	\author[Trinh Tuan]{Trinh Tuan}
	\thanks{E-mail: \textit{\bf tuantrinhpsac@yahoo.com}}

	\maketitle	
\begin{center}
Department of Mathematics, Electric Power University,\\
235-Hoang Quoc Viet Road, Hanoi, Vietnam.\\
E-mail: \textit{tuantrinhpsac@yahoo.com}
\end{center}
\begin{abstract}
The paper deals with $L_p$-boundedness of the Hartley-Fourier convolutions operator and their applied aspects. We establish various new Young-type inequalities and obtain the structure of a normed ring in Banach space when equipping it with such convolutional multiplication. Weighted $L_p$-norm inequalities of these convolutions are also considered. As applications, we investigate the solvability and the bounded $L_1$-solution  of a class of Fredholm-type integral equations and linear Barbashin's equations with the help of factorization identities of such convolutions. Several examples are provided to illustrate the obtained results to ensure their validity and applicability.
\vskip 0.3cm

\end{abstract}
\section{Introduction}
Integral inequalities are fundamental tools in the study of qualitative as well as quantitative properties of integral transformations and solutions of differential equations. In particular, the convolution inequality is indispensable because there are so many integral transformations and solutions of differential equations that are represented as convolutions (see \cite{Debnath2006Bhatta}). Certainly among all the convolutional transformations, the best known is the Fourier convolution. In 1912, Young. W. H introduced Young's inequality for the bounded estimation of the Fourier convolution (refer \cite{YoungWH1912}), which is 
$
\lv (f\underset{F}{*}g)\lv_{L_r (\R) }\ \leq \lv f\lv_{L_p (\R)}\ \lv g\lv_{L_q (\R)},
$ where $p, q, r >1$ such that $\frac{1}{p} +\frac{1}{q}=1+\frac{1}{r}$. R.A. Adams  and J.F. Fournier generalized Young's inequality for the Fourier convolution (refer to Theorem 2.24, page 33, in \cite{AdamsFournier2003sobolev}) to include a weight
$$
\bigg|\int\limits_{\R^n} (f\underset{F}{*}g)(x).\ \omega(x) dx \bigg|\leq \lv f\lv_{L_p (\R^n)} \lv g\lv_{L_q (\R^n)}\lv \omega\lv_{L_r (\R^n)},
$$
where  $p, q, r$ are real numbers in $(1,\infty)$ such that $\frac{1}{p} +\frac{1}{q}+\frac{1}{r}=2$ and $f \in L_p (\R^n), g\in L_q (\R^n), \omega \in L_r (\R^n)$ with the Fourier convolution defined by 
$(f \underset{F}{*}g)(x):=\frac{1}{\sqrt{2\pi}}\int\limits_{\R}f(y)\ g(x-y)dy,\ x\in \R$, then Young's inequality is seen as a consequence of Young's theorem  \cite{Sneddon1972}.   Nevertheless, Young's inequality does not hold for the important case $f, g \in L_2 (\R)$. Based on the technique in \cite{AdamsFournier2003sobolev,hoangtuan2017thaovkt,tuan2020Ukrainian} and derived from the previous results for Hartley--Fourier convolutions \cite{giang2009MauTuan,ThaoHVA2014MMA,Tuan2022MMA}, we  study the  $L_p$ norm estimates of these convolutions and applicabilities, which are also the main contributions of this paper. The obtained results are new Young type inequalities with their important applications in solving classes of  integro-differential equations. The applications aspect will be detailed here, for a large class of integral operators, with the results that will exemplify three different topics:

 \textbf{(i)} Non-commutative  normed ring structure on Banach space $L_1 (\R)$;
 
\textbf{(ii)} Solvability and boundedness of $L_1$-solutions of Fredholm's integral equation of the second kind;

\textbf{(iii)} Solvability and boundedness of $L_1$-solutions for Barbashin's equations \& Cauchy-type problem.\\ In addition, it should be noted that the factorization property of Hartley-Fourier convolutions is crucial in solving corresponding convolution-type equations \cite{Bracewell1986,ThaoHVA2014MMA,Tuan2022MMA}. It is also clear that convolution-type equations are very often used in the modeling of a broad range of different
problems \cite{castro2019NMTuan,castro2013,tuan2022Mediterranean}, so additional knowledge about their solvability is very welcome.

This paper is divided into five sections including the introduction and the remaining four parts. Section 2 is devoted to the presentation of the notions of Fourier and Hartley's transforms, and the relationship between them.  We also recall the definitions of Hartley, Hartley-Fourier convolutions and some results for factorization properties of these convolutions. We also prove that the assertion of Wiener-L\'evy's theorem holds for the Hartley transform (see Lemma \ref{WienerLevyforHartley}). Section 3 consists of three subsections. In Subsection 3.1, we prove Young's Theorem of the Hartley-Fourier generalized convolution \eqref{eq2.4} and corollaries for estimation in $L_1(\R)$ space. The general formulation of  Young-type inequality for the Hartley convolutions \eqref{eq2.5} is stated in Subsection 3.2. In Subsection 3.3, using H\"older's inequality, Fubini's theorem, and integral transforms, we establish norm inequalities in weighted $L_p (\R, \rho_j)$ spaces of each above convolutions. Section 4 deals with applications of the constructed convolutions. In Subsection 4.1, the linear space $L_1(\R)$,
equipped with the above-mentioned convolutions, becomes a normed ring, having no unit, which is non-commutative, and could be used in the theories of Banach algebra. Subsections 4.2 and 4.3 contain the most important results of this section where Fredholm's integral equations of the second kind and Barbashin's equation are considered simultaneously. By using the constructed convolutions together with the help of Wiener-L\'evy theorem and Schwartz's function classes, we provide necessary and sufficient conditions for the solvability of the integral equations of convolution type and obtain explicit $L_1$-solutions. The most obvious differences when solving the problems of subsections 4.2 and 4.3 in $L_1(\R)$ and the previous results in \cite{tuan2021korean,Tuan2022MMA} are that the  Fourier, Hartley transforms $L_2 (\R)\longleftrightarrow L_2 (\R)$ are an isometric (unitary) isomorphism \cite{Sogge1993fourier}, meanwhile it is no longer true for $L_1 ( \R) $. Furthermore, in the space $L_1(\R)$, Theorem (68, page 92, in \cite{Titchmarsh1986}) is no longer true. To overcome that, we use simultaneously the density in $L_1 (\R)$ of Schwartz's function classes \cite{WRudin1987}, Wiener-L\'evy Lemma \ref{WienerLevyforHartley}, and factorization properties. Subsection 4.4 is concerned with the Cauchy-type problem. Techniques used here to check if the solution of the problem satisfies the initial value theorem or not come from \cite{NaimarkMA1972} and \cite{Winer1932}. Namely, we follow closely the strategy  of Wiener-Tauberian's theorem. Some computational examples can be found in the final section of  the paper to illustrate the obtained results  to ensure their validity and applicability.

\section{Preliminaries}
We recall some notions and results coming from \cite{bateman1954,Bracewell1986,giang2009MauTuan,NaimarkMA1972,Poularikas1996,ThaoHVA2014MMA,Titchmarsh1986,tuan2022Mediterranean,Tuan2022MMA}.
We denote by (F) the Fourier integral transform (\cite{bateman1954,Titchmarsh1986}) of a function $f$ defined by
$(Ff)(y)=\frac{1}{\sqrt{2\pi}}\int\limits_{\R} e^{-ixy} f(x)dx$ with $y \in \R,$ where $e^{-ixy}=\frac{1+i}{2}\operatorname{Cas}(xy)+\frac{1-i}{2}\operatorname{Cas}(-xy).$ The Hartley transform was proposed as an alternative to the Fourier transform by R. V. L. Hartley in 1942 (see \cite{Bracewell1986}). The Hartley transform is an alternate means of analyzing
a given function in terms of its sinusoids. This transform is its own inverse and an efficient computational tool
in case of real-value functions. The Hartley transform
of a function is a spectral transform and can be obtained
from the Fourier transform by replacing the exponential
kernel $exp(-ixy)$ by $\operatorname{Cas}(-xy)$. The Hartley transform of a function $f$ can be expressed as either
\begin{equation}\label{eq2.1}
\big(H_{\left\{\begin{smallmatrix} 
	1\\ 
	2
	\end{smallmatrix} \right\}}f\big)(y)=\frac{1}{\sqrt{2\pi}} \int_{\R}
f(x) \operatorname{Cas}(\pm xy)dx, \ y \in \R, \end{equation}	
where $\operatorname{Cas}(\pm xy)=\cos(xy) \pm \sin (xy)$. Thus, it is obvious that we have the following correlation between the Hartley transform and the Fourier transform as follows
\begin{equation}\label{eq2.2}
\big(H_{\left\{\begin{smallmatrix} 
	1\\ 
	2
	\end{smallmatrix} \right\}}f\big)(y)= F\left(\dfrac{1\pm i}{2}f(x)+\dfrac{1\mp i}{2}f(-x)\right)(y),
\end{equation} and
\begin{equation}\label{eq2.3}
(Ff)(y)=H_{\left\{\begin{smallmatrix} 
	1\\ 
	2
	\end{smallmatrix} \right\}}\left(\dfrac{1-i}{2}f(\pm x)+\dfrac{1+i}{2} f(\mp x) \right)(y).
\end{equation}
Actually, the Hartley transform can be applied to
approximate the continuous transformation of a non-periodic signal of finite duration. Being a real-valued function, it has some computational advantages such as memory efficiency and faster harmonic solutions over the Fourier transform which is a complex tool, readers refer to details in \cite{Bracewell1986,Poularikas1996}. Consequently, the
method based on the Hartley transform is helpful to estimate harmonic effects for inrush current by giving a unified iterative solution that exhibits good convergence.
\begin{definition}
	The following generalized convolution related to  Hartley-Fourier $(H_1 ,F)$ transforms was introduced
	recently (see \cite{ThaoHVA2014MMA,Tuan2022MMA}) defined by
	\begin{equation}\label{eq2.4}
	\big(f\underset{H_1,F}{\ast}g\big)(x):=\frac{1}{2\sqrt{2\pi}}\int_{\R} g(y)[f(x+y)+f(x-y)+if(-x-y)-if(-x+y)]dy,\quad \forall x\in\mathbb{R}.
	\end{equation}For brevity, we call \eqref{eq2.4} it the Hartley-Fourier generalized convolution.\end{definition}
\begin{definition}\label{dnTichchapHartley}
	The convolution of two functions $f$ and $g$ for the Hartley $H_{\left\{\begin{smallmatrix} 
		1\\ 
		2
		\end{smallmatrix} \right\}}$  transform (see \cite{giang2009MauTuan,Tuan2022MMA}) is defined by 
	\begin{equation}\label{eq2.5}
	\big(f\underset{H_{\left\{\begin{smallmatrix} 
			1\\ 
			2
			\end{smallmatrix} \right\}}}{\ast} g\big) (x):= \frac{1}{2\sqrt{2\pi}}\int_{\R} f(y)[g(x+y) + g(x-y) + g(-x+y) - g(-x-y)] dy,\quad  \forall x\in\mathbb{R}.
	\end{equation}For brevity, we call \eqref{eq2.5} it the Hartley convolution.\end{definition}
\begin{proposition}
	\textup{i)  \cite{ThaoHVA2014MMA}.} Let $f,g$ be the functions belonging to $L_1 (\R)$, then Hartley-Fourier generalized convolution is $ \big(f\underset{H_1,F}{\ast}g\big)(x) \in L_1 (\R)$. Furthermore, the factorization equality \eqref{eq2.6} is valid.
	\begin{equation}\label{eq2.6}
	H_1\big(f\underset{H_1,F}{\ast}g\big)(y)=(H_1f)(y)(Fg)(y),\quad y\in\mathbb{R}.
	\end{equation}
	\textup{ii) \cite{giang2009MauTuan,Tuan2022MMA}.} Suppose that $f,g \in L_1 (\R)$,  then Hartley convolution \eqref{eq2.5} is well-defined, which  means that $ \big(f\underset{H_{\left\{\begin{smallmatrix} 
			1\\ 
			2
			\end{smallmatrix} \right\}}}{\ast} g\big)$ belongs to $L_1 (\R)$. And the following factorization equality 
	\begin{equation}\label{eq2.7}
	H_{\left\{\substack{1\\2}\right\}}\big(f\underset{H_{\left\{\substack{1\\2}\right\}}}{\ast} g\big) (y)=\big(H_{\left\{\substack{1\\2}\right\}}f\big)(y)\big(H_{\left\{\substack{1\\2}\right\}} g\big) (y),\quad y\in\mathbb{R},
	\end{equation} holds true for any $f,g \in L_1 (\R)$.
\end{proposition}
We need the following auxiliary lemmas. The Wiener-Levy's Theorem (refer \cite{NaimarkMA1972}) for the Fourier transform stated that if $\xi \in L_1(\R)$, then $1+(F\xi)(y)\neq 0$ for any $y \in \R$ is a necessary and sufficient condition for the existence of a function $k $  belonging to $L_1 (\R)$ such that  $(Fk)(y)=\frac{(F\xi)(y)}{1+(F\xi)(y)}$, where $\xi \in L_1(\R)$.  This theorem still holds true for the Hartley transform and we will detail it below.
\begin{lemma}\label{WienerLevyforHartley}\textup{(\textbf{Wiener-L\'evy's Theorem for Hartley transform})} For the Hartley transform is defined by \eqref{eq2.1} and suppose that 
	$\varphi \in L_1(\mathbb R)$, then $1+\big(H_{\left\{\substack{1\\2}\right\}} \varphi\big)(y)\ne 0$, $\forall y\in\mathbb R$ is a necessary and sufficient condition for the existence of a function $\ell\in L_1(\mathbb R)$ such that
	$
	\big(H_{\left\{\substack{1\\2}\right\}} \ell\big)(y)=\frac{\big(H_{\left\{\substack{1\\2}\right\}} \varphi\big)(y)}{1+\big(H_{\left\{\substack{1\\2}\right\}}\varphi\big)(y)},
	$ for any $y \in \R$.
\end{lemma}
\begin{proof}
We put                               
	$
	\xi(x):=\frac{1\pm i}{2} \varphi(x) + \frac{1\mp i}{2}\varphi(-x).
	$
	This implies that $\xi\in L_1(\mathbb R)$  if and only if $\varphi$ belongs to $L_1(\mathbb R)$. Hence, from the correlation between Hartley and Fourier transforms \eqref{eq2.2}, we obtain
	$
	(F\xi)(y)=\big(H_{\left\{\substack{1\\2}\right\}}\varphi\big) (y), \forall y\in\mathbb R.
	$ Applying Wiener--Levy's theorem to the Fourier (F) transform (see\cite{NaimarkMA1972}), we have $1+(F\xi)(y)\ne 0$,  which implies that $1+\big(H_{\left\{\substack{1\\2}\right\}}\varphi\big)(y)\ne 0$, $\forall y\in\mathbb R$ is a necessary and sufficient condition for the existence of a function $k\in L_1(\mathbb R)$ such that
	\begin{equation}\label{1}
	(Fk)(y) = \frac{(F\xi)(y)}{1+(F\xi)(y)} = \frac{\big(H_{\left\{\substack{1\\2}\right\}}\varphi\big)(y)}{1+\big(H_{\left\{\substack{1\\2}\right\}}\varphi\big)(y)},\forall y\in\mathbb R.
	\end{equation}
	\noindent Conversely, if set $
	\ell(x):= \frac{1-i}{2}k(\pm x) + \frac{1+i}{2}k(\mp x),
	$ by the same above argument, we deduce $\ell$ belongs to the $L_1(\mathbb R)$ if and only if $k\in L_1(\R)$ and by formula \eqref{eq2.3}, we obtain 
	\begin{equation}\label{2}
	\big(H_{\left\{\substack{1\\2}\right\}} \ell\big)(y)=(Fk)(y),\forall y\in\R.
	\end{equation}
	Coupling \eqref{1} with \eqref{2}, we have $
	\big(H_{\left\{\substack{1\\2}\right\}} \ell\big)(y)=(Fk)(y) = \frac{(F\xi)(y)}{1+(F\xi)(y)}=\frac{\big(H_{\left\{\substack{1\\2}\right\}}\varphi\big)(y)}{1+\big(H_{\left\{\substack{1\\2}\right\}}\varphi\big)(y)},\forall y\in \R.
	$
\end{proof}

\begin{lemma}\label{lemma4.1}
	Suppose that $g\in C^2(\mathbb R)\cap L_1(\mathbb R)$ such that $g', g''\in L_1(\mathbb R)$ and $\lim\limits_{|x|\to \infty} g(x)=\lim\limits_{|x|\to \infty} g'(x)=0$. For any function $f$ belongs to $L_1(\mathbb R)$, we obtain
	\begin{equation}\label{42}
	H_{\left\{\substack{1\\2}\right\}}\bigg\{\left(1-\frac{d^2}{dx^2}\right) \big(f\underset{H_{\left\{\substack{1\\2}\right\}}}{\ast}g\big)(x)\bigg\} (y) = (1+y^2) H_{\left\{\substack{1\\2}\right\}} \big(f \underset{H_{\left\{\substack{1\\2}\right\}}}{\ast}g\big) (y),
	\end{equation}
\end{lemma}
\begin{proof}
	By \eqref{eq2.1} and assumption of $g(x)$, we have	
	\begin{eqnarray}\label{4.3}
	\left(H_{\left\{\substack{1\\2}\right\}}g''(x)\right)(y) &=& \frac{1}{2\sqrt{\pi}}\int_{\R} g''(x) \mathrm{Cas}(\pm xy) dx,\quad y\in \mathbb R\nonumber\\
	&=&\frac{1}{\sqrt{2\pi}}\bigg\{ g'(x) \mathrm{Cas}(\pm xy)\Big|_{-\infty}^\infty \pm y\mathrm{Cas}(\pm xy)g(x)\Big|_{-\infty}^\infty -y^2\int_{\R} g(x)\mathrm{Cas}(\pm xy) dx\bigg\} \nonumber\\
	&=&-y^2\big(H_{\left\{\substack{1\\2}\right\}} g(x)\big) (y).
	\end{eqnarray}
	Since $f, g, g',$ and $g''$ are functions belonging to $L_1 (\R)$, we infer that the following convolutions $\Big(f\underset{H_{\left\{\substack{1\\2}\right\}}}{\ast} g\Big)$; $\Big(f \underset{H_{\left\{\substack{1\\2}\right\}}}{\ast} g'\Big)$; and $\Big(f \underset{H_{\left\{\substack{1\\2}\right\}}}{\ast} g''\Big)$ both belong to $L_1(\mathbb{R})$ (see \cite{ThaoHVA2014MMA,Tuan2022MMA}).  Notation $\mathcal{S}$ is the space of functions $g(x) \in C^{\infty} (\R)$ whose derivative decreases rapidly when $|x|$ tends to $\infty$, it is also known as Schwartz space. Then, the closure of $\mathcal{S} = L_1(\R)$, which means that $\mathcal{S}$ is the dense set in the $L_1 (\R)$ space. Thus, for any $g, g',$ and  $g''$ belong to the $L_1 (\R)$  there exists a sequence of functions $\{ g_n\} \in S$ such that $\{ g_n\}\rightarrow g,\ \{g'_n\} \rightarrow g'$ and $\{g''_n\}\rightarrow g'' $ when n tends to $\infty$. With the above function classes,  together with the assumption that  $g$ is a function belonging to the $L_1 (\R)$ space, the integral in the formula \eqref{eq2.5} is convergent. We continue to change the order of the integration and the differentiation as follow
	$$\begin{aligned}
	\frac{d^2}{dx^2}\big(f \underset{H_{\left\{\substack{1\\2}\right\}}}{\ast} g_n\big) (x)&= \frac{1}{2\sqrt{2\pi}} \int_{-\infty}^\infty f(y)\frac{d^2}{dx^2} \left\{g_n(x+y) + g_n(-x+y)
	g_n(x-y) - g_n(-x-y)\right\} dy\\
	&= \frac{1}{2\sqrt{2\pi}}\int_{-\infty}^\infty f(y)\left\{g''_n(x+y)+g''_n(-x+y)+ g''_n(x-y) - g''_n(-x-y)\right\} dy\\
	&= \big(f\underset{H_{\left\{\substack{1\\2}\right\}}}{\ast} g''_n\big)(x) \in L_1(\mathbb R).
	\end{aligned}$$
	Therefore 
	$
	\lim\limits_{n\to\infty} \frac{d^2}{dx^2}\big(f\underset{H_{\left\{\substack{1\\2}\right\}}}{\ast} g_n\big) (x)=\lim\limits_{n\to \infty} \big(f \underset{H_{\left\{\substack{1\\2}\right\}}}{\ast} g''_n\big)(x).
	$ This implies that $
	\frac{d^2}{dx^2}\big(f \underset{H_{\left\{\substack{1\\2}\right\}}}{\ast} g\big)(x) = \big(f \underset{H_{\left\{\substack{1\\2}\right\}}}{\ast} g''\big) (x)$ belongs to $L_1(\mathbb R).
	$ Coupling factorization equality \eqref{eq2.7} with \eqref{4.3}, we have 
	$$\begin{aligned}
	H_{\left\{\substack{1\\2}\right\}} \bigg\{\frac{d^2}{dx^2}\big(f \underset{H_{\left\{\substack{1\\2}\right\}}}{\ast} g\big)(x)\bigg\}(y)&=H_{\left\{\substack{1\\2}\right\}}\big(f \underset{H_{\left\{\substack{1\\2}\right\}}}{\ast}g''\big) (y)
	=\big(H_{\left\{\substack{1\\2}\right\}}f\big)(y)\big(H_{\left\{\substack{1\\2}\right\}}g''\big)(y)\\
	&=-y^2\big(H_{\left\{\substack{1\\2}\right\}}f\big)(y)\big(H_{\left\{\substack{1\\2}\right\}} g\big)(y)\\
	&= -y^2H_{\left\{\substack{1\\2}\right\}}\big(f \underset{H_{\left\{\substack{1\\2}\right\}}}{\ast} g\big) (y).
	\end{aligned}$$
	From the above equality, we infer 
	the desired conclusion of the lemma.
\end{proof}

Throughout the article, we shall make frequent use of the weighted Lebesgue spaces $L_p (\R , w(x)dx)$, $1\leq p < \infty$ with respect to a positive measure $w(x)dx$ equipped with the norm for which $\lv f \lv_{L_p (\R , w)}=\bigg( \int\limits_\R |f(x)|^p w(x)dx \bigg)^{1/p}<+ \infty.$ If the weighted function $w=1$, then $L_p (\R ,w)\equiv L_p (\R)$. In case $p=\infty$, then the norm of functions is defined by $\lv f \lv_{L_{\infty}(\R)}= \sup\limits_{x\in \R} |f(x)|< +\infty$.

\section{$L_p$-boundedness for the Hartley-Fourier convolutions}

\subsection{Young's theorem for Hartley-Fourier generalized convolution}
\begin{theorem}[\textbf{Young-type Theorem for Hartley-Fourier generalized convolution}]\label{theorem1}
	Let $p,q$, and $r$ be real numbers in open interval $(1,\infty)$ such that $\frac{1}{p}+\frac{1}{q}+\frac{1}{r}=2$.	
	For any functions $g\in L_p(\mathbb{R})$, $f\in L_q(\mathbb{R})$, and $h\in L_r(\mathbb{R})$,  we obtain the following inequality 
	
	\begin{equation}\label{eq3.1}
	\bigg|\int_{\R} \big(f\underset{H_1 ,F}{*} g\big)(x)\cdot h(x)\, dx\bigg|\leq \sqrt{\frac{2}{\pi}} \|g\|_{L_p(\mathbb{R})} \|f\|_{L_q(\mathbb{R})}\|h\|_{L_r(\mathbb{R})},
	\end{equation}
	where $\big(f\underset{H_1 ,F}{*} g\big)$ is defined by \eqref{eq2.4}.
\end{theorem}
\begin{proof}
	Let  $p_1,q_1,r_1$  be the conjugate exponentials of $p,q,r$, respectively. This means that  $\frac{1}{p}+\frac{1}{p_1}=\frac{1}{q}+\frac{1}{q_1}=\frac{1}{r}+\frac{1}{r_1}=1$, together with the assumption of theorem, we get the correlation between exponential numbers as follows
	\begin{equation}\label{eq3.2}
	\left\{\begin{array}{l}
	\frac{1}{p_1}+\frac{1}{q_1}+\frac{1}{r_1}=1,\\
	p\left(\frac{1}{q_1}+\frac{1}{r_1}\right)=q\left(\frac{1}{p_1}+\frac{1}{r_1}\right)=r\left(\frac{1}{p_1}+\frac{1}{q_1}\right)=1.
	\end{array}\right.
	\end{equation}
	For simplicity, we set
	\begin{eqnarray*}
		A_1(x,y)&=&\left|f(x+y) + f(x-y) + if(-x-y) - if(-x+y)\right|^{\frac{q}{p_1}}\cdot|h(x)|^{\frac{r}{p_1}}\in L_{p_1}(\R^2),\\
		A_2(x,y)&=& |f(x+y) + f(x-y) + if(-x-y) - if(-x+y)|^{\frac{q}{r_1}} |g(x)|^{\frac{p}{r_1}} \in L_{r_1}(\R^2),\\
		A_3(x,y)&=& |g(y)|^{\frac{p}{q_1}} |h(x)|^{\frac{r}{q_1}}\in L_{q_1}(\R^2).
	\end{eqnarray*}
	From \eqref{eq3.2}, we deduce that
	\begin{equation}\label{eq3.3}
	\prod_{i=1}^3 A_i(x,y)=|f(x+y) + f(x-y) + if(-x-y) - if(-x+y)|\ |g(x)| |h(x)|,\quad \forall (x,y)\in \R^2.
	\end{equation}
	On the other hand, since $q>1$, then $t^q$ is a convex function, using
	the change of variables theorem, we obtain
	\begin{eqnarray*}
		&&\int_{\R} |f(x+y) + f(x-y) + if(-x-y) - if(-x+y)|^q dy\\ &&\leq 4^{q-1}\int_{\R}\left(|f(x+y)|^q + |f(x-y)|^q + |if(-x-y)|^q + |if(-x+y)|^q\right) dy=4^q\int_{\R} |f(t)|^q dt.
	\end{eqnarray*}
	Based on the assumption of $f\in L_q (\R), h\in L_r (\R)$, using Fubini's theorem,  we obtain  $L_{p_1} (\R^2)$-norm estimation for the operator $A_1 (x,y)$ as follows
	$$\begin{aligned}
		\|A_1\|^{p_1}_{L_{p_1}(\R^2)} &= \int_{\R^2} |f(x+y) + f(x-y) + if(-x-y) - if(-x+y)|^q |h(x)|^r dxdy\nonumber\\
		&\leq 4^q\int_{\R} \biggl\{\int_{\R}|f(t)|^q dt\biggr\}|h(x)|^r dx\nonumber
		= 4^q\bigg(\int_{\R} |f(t)|^q dt\bigg)\bigg(\int_{\R} |h(x)|^r dx\bigg)\nonumber=4^{q}\| f\|^{q}_{L_q(\mathbb{R})} \|h\|^{r}_{L_r(\mathbb{R})}.\end{aligned}$$
	This means that\begin{equation}\label{eq3.4}
	\|A_1\|_{L_{p_1}(\R^2)}= 4^{\frac{q}{p_1}}\| f\|^{\frac{q}{p_1}}_{L_q(\mathbb{R})} \|h\|^{\frac{r}{p_1}}_{L_r(\mathbb{R})}.
	\end{equation}
	Similar to what we did with the evaluation \eqref{eq3.4} of $A_1 (x,y)$, we also get the norm estimation of $A_2$ on $L_{r_1} (\R^2)$ as follows
	\begin{equation}\label{eq3.5}
	\|A_2\|_{L_{r_1}(\R^2)} = 4^{\frac{q}{r_1}} \|f\|_{L_q(\mathbb{R})}^{\frac{q}{r_1}} \|g\|_{L_p(\mathbb{R})}^\frac{p}{r_1}.
	\end{equation}
	And $L_{q_1} (\R^2)$-norm estimation for the operator $A_3$ has the following form
	$$
		\|A_3\|_{L_{q_1}(\R^2)}^{q_1}= \int_{\R^2} |g(y)|^p |h(x)|^r dxdy=\bigg(\int_{\R} |g(y)|^p dy\bigg)\bigg(\int_{\R} |h(x)|^r dx\bigg)=\|g\|^p_{L_p(\mathbb{R})} \|h\|^r_{L_r(\mathbb{R})}.
	$$
	Therefore, we have \begin{equation}\label{eq3.6}
	\|A_3\|_{L_{q_1}(\R^2)}=\|g\|^{\frac{p}{q_1}}_{L_p(\mathbb{R})} \|h\|^{\frac{r}{q_1}}_{L_r(\mathbb{R})}.
	\end{equation}
	Combining  \eqref{eq3.4}\eqref{eq3.5} and  \eqref{eq3.6}, under condition \eqref{eq3.2}, we obtain\begin{equation}\label{eq3.7}
	\|A_1\|_{L_{p_1}(\R^2)} \|A_2\|_{L_{r_1}(\R^2)} \|A_3\|_{L_{q_1}(\R^2)} \leq 4\|g\|_{L_p(\R)} \|f\|_{L_q(\R)} \|h\|_{L_r(\R)}.
	\end{equation}
	Furthermore, from \eqref{eq2.4} and \eqref{eq3.3} we have
	$$\begin{aligned}
		\bigg|\int_{\R} \big(f \underset{H_1,F}{\ast} g\big)(x) \cdot h(x) dx\bigg|&\leq \frac{1}{2\sqrt{2\pi}}\int_{\R^2} |g(y)| |f(x+y) + f(x-y) +if(-x-y) - if(-x+y)| |h(x)| dydx\\
		&= \frac{1}{2\sqrt{2\pi}} \int_{\R^2}\ \prod_{i=1}^3 A_i(x,y) dxdy
	\end{aligned}$$
Thus $\frac{1}{p_1}+\frac{1}{q_1}+\frac{1}{r_1}=1$. Using H\"older inequality and \eqref{eq3.7}, then $\forall g\in L_p(\mathbb{R}),~ f\in L_q(\mathbb{R}),~h\in L_r(\mathbb{R})$ we have
$$\begin{aligned}
&\frac{1}{2\sqrt{2\pi}} \int_{\R^2}\ \prod_{i=1}^3 A_i(x,y) dxdy\\&\leq \frac{1}{2\sqrt{2\pi}}\left(\int_{\R^2}|A_1 (x,y) |^{p_1} dxdy\right)^{1\textfractionsolidus p_1}\left(\int_{\R^2}|A_2 (x,y) |^{r_1} dxdy\right)^{1\textfractionsolidus r_1}\left(\int_{\R^2}|A_3 (x,y) |^{q_1} dxdy\right)^{1\textfractionsolidus q_1}\\
&=\frac{1}{2\sqrt{2\pi}}\|A_1\|_{L_{p_1}(\R^2)} \|A_2\|_{L_{r_1}(\R^2)} \|A_3\|_{L_{q_1}(\R^2)}
\leq \sqrt{\frac{2}{\pi}} \|g\|_{L_p(\mathbb{R})} \|f\|_{L_q(\mathbb{R})} \|h\|_{L_r(\R)}
\end{aligned}$$
\end{proof}
\begin{remark}
	\noindent Readers can find another result similar to this theorem through Theorem 2.4 in \cite{tuan2022Mediterranean}.
\end{remark}
\noindent In case the given function $h(x)$ becomes Hartley-Fourier generalized convolution operator \eqref{eq2.4}, the following Young-type inequality is a direct consequence of Theorem \ref{theorem1}.
\begin{corollary}[\textbf{Young's inequality for Hartley-Fourier generalized convolution}]\label{hequa1}
	Let $p,q,r > 1$ be real numbers, satisfying  $\frac{1}{p}+\frac{1}{q}=1+\frac{1}{r}$. If $g\in L_p(\mathbb{R})$, $f\in L_q(\mathbb{R})$, then the convolution \eqref{eq2.4} is well-defined and belongs to $L_r(\mathbb{R})$. Moreover, the following inequality holds
	\begin{equation}\label{eq3.8}
	\big\| f\underset{H_1, F}{*} g\big\|_{L_r(\mathbb{R})} \leq \sqrt{\frac{2}{\pi}} \|g\|_{L_p(\mathbb{R})} \|f\|_{L_q(\mathbb{R})}.
	\end{equation}
\end{corollary}
\begin{proof}
	Let $r_1$ be the conjugate exponent of $r$, i.e $\frac{1}{r}+\frac{1}{r_1}=1$. From the assumptions of Corollary \ref{hequa1}, we have $\frac{1}{p}+\frac{1}{q}+\frac{1}{r_1}=2$,  which shows the numbers $p$, $q$, and $r_1$  satisfy the conditions of  Theorem \ref{theorem1} (with  $r$  being replaced by $r_1$). Therefore, if $g\in L_p(\mathbb{R})$, $f\in L_q(\mathbb{R})$, then the linear operator	$$Lh:=\int_{\R} \big(f \underset{H_1, F}{*} g\big) (x)\cdot h(x) dx$$ is bounded in $L_{r_1}(\R)$.
	Consequently, by the Riesz's representation theorem \cite{Stein1971Weiss}, then generalized convolution $\big(f \underset{H_1, F}{*} g\big)$ belongs to $L_r(\mathbb{R})$.
	To prove the inequality \eqref{eq3.8},  we choose the function $$h(x):=\mathrm{sign}\bigg\{ \big(f \underset{H_1, F}{*} g\big)(x)\bigg\}^r \times \bigg\{ \big(f \underset{H_1, F}{*} g\big)(x)\bigg\}^{\frac{r}{r_1}}.$$ 
	Then $h \in L_{r_1} (\R)$, with the norm $\|h\|_{L_{r_1}(\mathbb{R})} = \big\|f \underset{H_1, F}{*}  g\big\|^{\frac{r}{r_1}}_{L_r(\mathbb{R})}.$ Applying inequality \eqref{eq3.1} to such  function $h(x)$, we get
	$$\begin{aligned}\label{eq2.9}
		\big\|f \underset{H_1, F}{*}  g\big\|_{L_r(\mathbb{R})}^r &= \int_{\R} \big|\big(f \underset{H_1, F}{*}  g\big)(x)\big|^r dx\nonumber
		= \bigg|\int_{\R} \big(f \underset{H_1, F}{*}  g\big)(x)\cdot h(x)\bigg|\nonumber\\
		&\leq \sqrt{\frac{2}{\pi}} \|g\|_{L_p(\mathbb{R})} \|f\|_{L_q(\mathbb{R})} \|h\|_{L_{r_1}(\mathbb{R})}\nonumber
		=\sqrt{\frac{2}{\pi}} \|g\|_{L_p(\mathbb{R})} \|f\|_{L_q(\mathbb{R})} \big\|f \underset{H_1, F}{*}  g\big\|^{\frac{r}{r_1}}_{L_r(\mathbb{R})}.
	\end{aligned}$$
	Since $r - \frac{r}{r_1} =1$, from the above equality, we arrive at the conclusion of the corollary.
\end{proof}

Notice that, inequalities \eqref{eq3.1} and \eqref{eq3.8} do not hold for the important case $f, g \in L_2 (\R)$.  We know that the problem is determining the duals of the space $L_p$ with $p \in$ $[1, \infty)$. There are basically two cases of this problem that are $p=1$, and the case $1<p<\infty$. The major difference between these two cases is the fact that for $1<p<\infty$, there is a \textquotedblleft nice\textquotedblright characterization of the dual of $L_p$ which identifies it with $L_q$ (with $1/p + 1/q=1$), and this identification holds without any restriction on the underlying space. Otherwise, the duality of $L_1$ will be defined by $L_{\infty}$ locally, but this identification will work only if the underlying measure space is decomposable. Now let us consider the case $p=q=r=1$, this means $f$ and $g$ are functions belonging to the $L_1 (\R)$. 
We deduce that $ \big(f\underset{H_1, F}{*}  g\big)$ belongs to $L_1 (\R)$ (see \cite{ThaoHVA2014MMA,Tuan2022MMA}) and get the estimate $L_1$-norm as follows.
\begin{theorem}
	For any $f, g\in L_1 (\R)$, then 
\begin{equation}\label{3.9}
\big\|f \underset{H_1, F}{*}  g\big\|_{L_1 (\R)} \leq \sqrt{\frac{2}{\pi}} \|f\|_{L_1(\mathbb{R})} \|g\|_{L_1(\mathbb{R})}.
\end{equation}\end{theorem}
\begin{proof}
	Indeed, by the Definition \ref{dnTichchapHartley}, we infer that
	$$\begin{aligned}
	\big\|f \underset{H_1, F}{*}  g\big\|_{L_1(\mathbb{R})} &\leq  \frac{1}{2\sqrt{2\pi}} \int_{\R^2} |g(y)|\ |f(x+y) + f(x-y)
	+ if(-x-y) - if(-x+y)| dydx\\
	&\leq \frac{1}{2\sqrt{2\pi}}\bigg\{ \int_{\R^2} |g(y)| |f(x+y)| dydx
	+\int_{\R^2} |g(y)|\ |f(x-y)| dydx\\&
	+ \int_{\R^2} |g(y)|\ |if(-x-y)| dydx
	+ \int_{\R^2} |g(y)|\ |if(-x+y)| dydx\bigg\}.
	\end{aligned}$$
	Since $f,g\in L_1 (\R)$, through the change of variables combined with Fubini's theorem, we obtain
	$$\begin{aligned}
	\big\|f \underset{H_1, F}{*}  g\big\|_{L_1 (\R)} &\leq \frac{4}{2\sqrt{2\pi}}\bigg(\int_{\R} |g(y)| dy\bigg) \bigg(\int_{\R} |f(t)| dt\bigg)= \sqrt{\frac{2}{\pi}} \|f\|_{L_1(\mathbb{R})} \|g\|_{L_1(\mathbb{R})}.
	\end{aligned}$$
\end{proof}
What about the case $r=\infty$?
\begin{theorem}
	Suppose that $p,q>1$ and satisfy $\frac{1}{p}+\frac{1}{q}=1$. If $g\in L_p (\R), f\in L_q (\R)$, then convolution operator \eqref{eq2.4} is a bounded function for any $x\in \R$. Moreover, the following inequality holds
	\begin{equation}\label{case=vocung}
	\lv f\underset{H_1, F}{*}g\lv_{L_\infty (\R)} \leq \sqrt{\frac{2}{\pi}} \|g\|_{L_p (\R)}\|f\|_{L_q (\R)}.
	\end{equation}
\end{theorem}
\begin{proof}
From \eqref{eq2.4}, we have 
$$|(f\underset{H_1 , F}{*}g)|\leq\frac{1}{2\sqrt{2\pi}}\int\limits_{\R} |g(y)|\ (|f(x+y)|+|f(x-y)|+|if(-x-y)|+|if(-x+y)|)dy.$$
Using H\"older's inequalities for the pair of conjugate exponents $p$ and $q$, we deduce that
$$|(f\underset{H_1 , F}{*}g)|\leq\frac{1}{2\sqrt{2\pi}}\left(\int_\R |g(\xi)|^p d\xi\right)^{\frac{1}{p}}\left(\int_\R 4^q|f(t)|^q dt\right)^{\frac{1}{q}}=\frac{4}{2\sqrt{2\pi}}\\|g\|_{L_p (\R)}\|f\|_{L_q (\R)}< \infty.$$
This implies that the convolution operator $f\underset{H_1 , F}{*}g$ is bounded function $\forall x\in \R$ and the desired conclusion inequality \eqref{case=vocung}.
\end{proof}
Even in the case of the classical Fourier convolution, one has only the Young's convolution \eqref{3.9} for the case $p=q=r = 1$. Therefore, the Young-type inequality \eqref{3.9} is
a specific characteristic of the $ \underset{H_1 ,F}{*} $ convolution  introduced here. Indeed, suppose that $f, g$ are functions belonging to $L_1 (\R)$ space, then the Fourier convolution  $(f\underset{F}{*} g) \in L_1 (\R)$ and   $L_1$-norm estimation inequality are as follows $\lv (f\underset{F}{*} g)\lv_{L_1 (\R)} \leq \lv f\lv_{L_1 (\R)} \lv g\lv_{L_1 (\R)},$ (see \cite{Sneddon1972}). It shows, on the $L_1 (\R)$ space, when being equipped with the multiplication defined as Fourier convolution multiplication, we get $\left(L_1 (\R), \underset{F}{*} \right)$ as a commutative normed ring \cite{Sneddon1972} for the aforementioned convolution multiplication. However, the ring structure is no longer commutative if the convolution multiplication is replaced by the Hartley-Fouier generalized convolution \eqref{eq2.4}, this will be demonstrated in detail in Subsection 4.1 by applying inequality \eqref{3.9}.
\subsection{Young's inequalities for Hartley convolution}
In a similar way, we obtain Young type theorem for the  Hartley convolution \eqref{eq2.5}. The results of this part are proved to be similar to those in Subsection 3.1, respectively. 
\begin{theorem}\label{thm3.4}
	Let $p, q, r >1$ be real numbers satisfying $\frac{1}{p}+\frac{1}{q}+\frac{1}{r}=2$. Then the following inequalities hold true for all $f\in L_p(\mathbb{R})$, $g\in L_q(\mathbb{R})$ and $h\in L_r(\mathbb{R})$.
	\begin{equation}\label{3.1}
	\bigg|\int_{\R} \big(f \underset{H_{\left\{\substack{1\\2}\right\}}}{*} g\big)(x)\cdot h(x) dx\bigg|\leq \sqrt{\frac{2}{\pi}}\|f\|_{L_p(\mathbb{R})} \|g\|_{L_q(\mathbb{R})} \|h\|_{L_r(\mathbb{R})},
	\end{equation}	
\end{theorem}

\begin{corollary} \label{cor3.2}
	If $p,q,r>1$ and satisfy $\frac{1}{p}+\frac{1}{q}=1+\frac{1}{r}$, for any $f\in L_p(\mathbb{R})$, $g\in L_q(\mathbb{R})$, then convolution $\big(f \underset{H_{\left\{\substack{1\\2}\right\}}}{*} g\big)$ belongs to $L_r(\mathbb{R})$. Then we obtain the following norm estimation
	\begin{equation}\label{3.11}
	\big\|f \underset{H_{\left\{\substack{1\\2}\right\}}}{*} g\big\|_{L_r(\mathbb{R})} \leq \sqrt{\frac{2}{\pi}}\|f\|_{L_p(\mathbb{R})} \|g\|_{L_q(\mathbb{R})}.
	\end{equation}
\end{corollary}
\noindent The proof is similar to Corollary \ref{hequa1}. And $L_1$-norm estimate of Hartley convolution \eqref{eq2.5} is as follows.
\begin{equation}\label{3.12}
\big\|f  \underset{H_{\left\{\substack{1\\2}\right\}}}{*} g \big\|_{L_1(\mathbb{R})} \leq \sqrt{\frac{2}{\pi}} \|f\|_{L_1(\mathbb{R})} \|g\|_{L_1(\mathbb{R})}\quad \forall f, g \in L_1 (\R).
\end{equation}
The same thing for the case $r=\infty$ is also obtained  with convolution \eqref{eq2.5}.
\subsection{Weighted $L_p$-boundedness for the Hartley convolutions}
By considering the $L_p$ norms in more naturally determined weighted spaces, S. Saitoh \cite{Saitoh2000}
gave Fourier convolution norm inequalities in the form

$$\big\lv\big( (F_1 \rho_1) \underset{F}{*}(F_2 \rho_2)\big) \cdot (\rho_1 \underset{F}{*} \rho_2)^{\frac{1}{p}-1}\big\lv_{L_p (\R)} \leq \big\lv F_1 \big\lv_{L_p (\R,|\rho_1|)} \big\lv F_2\big\lv_{L_p (\R,|\rho_2|)},\quad p>1,$$ where $\rho_j $ are non-vanishing functions, $F_j \in L_p (\R, |\rho_j|),\, j=1,2$. Here, the norm of $F_j$ in the weighted space $L_{p}(\mathbb{R}, \rho_j)$ is understood as $
\big\lv F_j  \big\lv_{L_{p}(\mathbb{R}, \rho_j)}=\bigg\{\displaystyle\int\limits_{\R}|F_j(x)|^{p} \rho_j(x) \mathrm{d} x\bigg\}^{\frac{1}{p}}.$ This type of inequality is very convenient as many applications require the \textquotedblleft same\textquotedblright $L_p$ norms. Saitoh's basic idea in weighted $L_p$ norm inequalities is fairly simple, however, it has been successfully implemented in many application \cite{Saitoh2000}. Saitoh's inequality also can be applied to estimating the solution to a parabolic integro-differential equation \cite{hoangtuan2017thaovkt} modeling a scattered acoustic field.  Based on the above aspects, together with using H\"older's inequality, and Fubini's theorem, we established another result in weighted space $L_{p}(\mathbb{R}, \rho_j)$ for Hartley convolution \eqref{eq2.5}. Some techniques used in the proof of our theorem come from \cite{tuan2020Ukrainian,Tuan2023VKT}, and we follow closely the strategy of these results.
\begin{theorem}\label{SaitohForHartley}
	 Suppose that $\rho_1, \rho_2$  are non-vanishing  positive functions such that convolution $(\rho_1 \underset{H_{\left\{\substack{1\\2}\right\}}}{*}\rho_2)$, given by \eqref{eq2.5}, is well-defined. For any functions $F_1 \in L_p (\R, \rho_1)$ and $F_2 \in L_p (\R, \rho_2)$ with $p >1$,  the following $L_p (\R)$-weighted inequality holds true 
	 	\begin{equation}\label{eq26}
	 \big\lv \big( F_1 \rho_1\underset{H_{\left\{\substack{1\\2}\right\}}}{*}F_2 \rho_2\big)  (\rho_1 \underset{H_{\left\{\substack{1\\2}\right\}}}{*}\rho_2)^{\frac{1}{p}-1} \big\lv_{L_p (\R)} \leq \sqrt{\frac{2}{\pi}} \big\lv F_1 \big\lv_{L_p (\R, \rho_1)} \big\lv F_2\big\lv_{L_p (\R, \rho_2)}.
	 \end{equation} 
\end{theorem}
\begin{proof}
	From \eqref{eq2.5}, we have
	\begin{equation}\label{eq3.14}
	\begin{aligned}
	&\big\lv \big( F_1 \rho_1\underset{H_{\left\{\substack{1\\2}\right\}}}{*}F_2 \rho_2\big)  (\rho_1 \underset{H_{\left\{\substack{1\\2}\right\}}}{*}\rho_2)^{\frac{1}{p}-1} \big\lv^p_{L_p (\R)}=\int_\R\bigg|\big(F_1 \rho_1\underset{H_{\left\{\substack{1\\2}\right\}}}{*}F_2 \rho_2\big)(x)  (\rho_1 \underset{H_{\left\{\substack{1\\2}\right\}}}{*}\rho_2)^{\frac{1}{p}-1}(x) \bigg|^p dx
	\\&=
	\frac{1}{2\sqrt{2\pi}}\times\int_\R\biggl\{\bigg|\int_\R (F_1 \rho_1)(y)\bigg((F_2 \rho_2)(x+y) + (F_2 \rho_2)(x-y)+(F_2 \rho_2)(-x+y)-(F_2 \rho_2)(-x-y)\bigg)dy \bigg|^p\\ &\times\bigg|\int_\R \rho_1 (y)\bigg(\rho_2 (x+y)+\rho_2 (x-y)+\rho_2 (-x+y)-\rho_2(-x-y) \bigg)dy\bigg|^{1-p}\biggr\}dx	\\
	&\leq \frac{1}{2\sqrt{2\pi}}\int_\R \big( \mathcal{A}_1 (x,y)+\mathcal{A}_2 (x,y)+\mathcal{A}_3 (x,y)+\mathcal{A}_4 (x,y) \big)\mathscr{T}^{1-p}_{ (x,y)}dx,\ p>1,
	\end{aligned}
\end{equation}
where we put
$$\begin{aligned}
&A_1 (x,y)=\int_\R |(F_1 \rho_1)(y)|\ |(F_2 \rho_2)(x+y)|dy,\\
&A_2 (x,y)=\int_\R |(F_1 \rho_1)(y)|\ |(F_2 \rho_2)(x-y)|dy,\\
&A_3 (x,y)=\int_\R |(F_1 \rho_1)(y)|\ |(F_2 \rho_2)(-x+y)|dy,\\
&A_4 (x,y)=\int_\R |(F_1 \rho_1)(y)|\ |(F_2 \rho_2)(-x-y)|dy.
\end{aligned}$$
And $$\mathscr{T}^{1-p}_{ (x,y)}=\biggl\{\int_\R \rho_1 (y)\big(\rho_2 (x+y)+\rho_2 (x-y)+\rho_2 (-x+y)+\rho_2(-x-y) \big)dy\biggr\}^{1-p}.$$  Using H\"older's inequalities for the pair of conjugate exponents $p$ and $q$ towards operators $A_i, i=\overline{1,4}$, we have
$$\begin{aligned}
&A_1 (x,y)\leq \biggl\{\int_\R |F_1 (y)|^p \rho_1 (y) |F_2 (x+y)|^p \rho_2(x+y)dy\biggr\}^{\frac{1}{p}}\times\biggl\{\int_\R\rho_1 (y)\rho_2 (x+y)dy\biggr\}^{\frac{1}{q}},\\
&A_2 (x,y)\leq \biggl\{\int_\R |F_1 (y)|^p \rho_1 (y) |F_2 (x-y)|^p \rho_2(x-y)dy\biggr\}^{\frac{1}{p}}\times\biggl\{\int_\R\rho_1 (y)\rho_2 (x-y)dy\biggr\}^{\frac{1}{q}},\\
&A_3 (x,y)\leq \biggl\{\int_\R |F_1 (y)|^p \rho_1 (y) |F_2 (-x+y)|^p \rho_2(-x+y)dy\biggr\}^{\frac{1}{p}}\times\biggl\{\int_\R\rho_1 (y)\rho_2 (-x+y)dy\biggr\}^{\frac{1}{q}},\\
&A_4 (x,y)\leq \biggl\{\int_\R |F_1 (y)|^p \rho_1 (y) |F_2 (-x-y)|^p \rho_2(-x-y)dy\biggr\}^{\frac{1}{p}}\times\biggl\{\int_\R\rho_1 (y)\rho_2 (-x-y)dy\biggr\}^{\frac{1}{q}}.
\end{aligned}$$
	Therefore, recalling that  $t^{\frac{1}{p}}, t^{\frac{1}{q}}\ (p,q>1)$ are concave functions, together with the above-estimates, we obtain
	\begin{equation}\label{eq3.19}
	\begin{aligned}
	&\sum_{i=1}^{4}A_i (x,y)\leq \biggl\{\int_\R |(F_1 \rho_1)(y)|^p \rho_1 (y) \bigg(|F_2 (x+y)|^p \rho_2 (x+y)+|F_2 (x-y)|^p \rho_2 (x-y)\\&+|F_2 (-x+y)|^p \rho_2 (-x+y)+|F_2 (-x-y)|^p \rho_2(-x-y) \bigg)dy \biggr\}^{\frac{1}{p}}\times \mathscr{T}^{\frac{1}{q}}_{ (x,y)}.\end{aligned}
	\end{equation}
	From \eqref{eq3.14} and \eqref{eq3.19}, we deduce that
	$$\begin{aligned}
	&\big\lv \big( F_1 \rho_1\underset{H_{\left\{\substack{1\\2}\right\}}}{*}F_2 \rho_2\big)  (\rho_1 \underset{H_{\left\{\substack{1\\2}\right\}}}{*}\rho_2)^{\frac{1}{p}-1} \big\lv^p_{L_p (\R)}\\&\leq \frac{1}{2\sqrt{2\pi}}\int_\R \biggl\{\int_\R |F_1 (y)|^p \rho_1 (y)\bigg( |F_2 (x+y)|^p \rho_2 (x+y) + |F_2(x-y)|^p \rho_2 (x-y)+|F_2 (-x+y)|^p \rho_2 (-x+y)\\&+|F_2 (-x-y)|^p \rho_2 (-x-y)\bigg)dy \times  \mathscr{T}^{\frac{p}{q}+1-p}_{ (x,y)} \biggr\}dx.
	\end{aligned}$$
	Since $p$ and $q$ are a conjugate pair ($\frac{1}{p}+\frac{1}{q}=1$), it implies that $\frac{p}{q}+1-p=0$, by using the Fubini theorem, we obtain
	$$\begin{aligned}
	\big\lv \big( F_1 \rho_1\underset{H_{\left\{\substack{1\\2}\right\}}}{*}F_2 \rho_2\big)  (\rho_1 \underset{H_{\left\{\substack{1\\2}\right\}}}{*}\rho_2)^{\frac{1}{p}-1} \big\lv^p_{L_p (\R)}&\leq \frac{4}{2\sqrt{2\pi}} \int_{\R^2} |F_1(y)|^p \rho_1 (y)\ |F_2(t)|^p \rho_2 (t)\ dydt\\
	&=\sqrt{\frac{2}{\pi}}\bigg(\int_\R |F_1 (y)|^p \rho_1 (y) dy\bigg) \bigg(\int_\R |F_2 (t)|^p \rho_2 (t) dt\bigg).
	\end{aligned}$$
	From the above equality, we infer 
	the desired conclusion of the theorem.
\end{proof}
Clearly, if $\rho_2 (x)$ is an identical function $1$ for all $x \in \R$ and $0<\rho_1 (x)\in L_1 (\R)$, then we have $$\big|(\rho_1 \underset{H_{\left\{\substack{1\\2}\right\}}}{*}1)(x)\big|=\frac{2}{2\sqrt{2\pi}}\int_\R \rho_1(x)dx=\frac{1}{\sqrt{2\pi}}\|\rho_1\|_{L_1 (\R)} < \infty.$$ This means that $(\rho_1 \underset{H_{\left\{\substack{1\\2}\right\}}}{*}1)$ is well-defined and consequently 
$$\big|(\rho_1 \underset{H_{\left\{\substack{1\\2}\right\}}}{*}1)(x)\big|^{1-\frac{1}{p}} \leq \left(\frac{1}{\sqrt{2\pi}}\right)^{1-\frac{1}{p}} \lv\rho_1\lv^{1-\frac{1}{p}}_{L_1 (\R)}.$$ Combining with Theorem \ref{SaitohForHartley}, we arrive at the following corollary
\begin{corollary}\label{HequaSaitoh}
	Let $\rho_1 $ be a  positive function belonging to $L_1 (\R)$. If  $F_1,F_2$ are functions belonging to  $L_p (\R , \rho_1)$ and $L_p (\R)$, respectively, with $p>1$, then we have the following estimate
	\begin{equation}\label{eq3.20}
\lv F_1 \rho_1  \underset{H_{\left\{\substack{1\\2}\right\}}}{*} F_2 \lv_{L_p (\R)} \leq \frac{1}{\pi} (\sqrt{2\pi})^{\frac{1}{p}}\lv\rho_1\lv^{1-\frac{1}{p}}_{L_1 (\R)} \lv F_1 \rho_1\lv_{L_p (\R, \rho_1)}  \lv F_2 \lv_{L_p (\R)}.
	\end{equation}
\end{corollary}
\noindent Note that, the inequalities obtained in Theorem \ref{SaitohForHartley} and Corollary \ref{HequaSaitoh} are still true for case $p=2$. In a similar way as above, we also obtained $L_p$-boundedness in weighted space $L_p (\R, \rho_j)$ for generalized convolution operator \eqref{eq2.4}.  
\section{Some applications}

\subsection{Structure Nomerd ring}
To avoid confusion with the notation of exponential function $(e)$, we denote  $\mathcal{U}$ as the unit element for the multiplicative on vector space $V$.
\begin{definition}
	\textup{(See \cite{NaimarkMA1972}). A vector space $V$ is with a ring structure and a vector norm is called a \textit{normed ring} if $\lv v \ w\lv \leq \lv v\lv \lv w\lv,$ for all $v, w \in V$. If $V$ has a multiplicative unit element, then it also requires that  $\lv \mathcal{U} \lv=1$. }
\end{definition}
\begin{theorem}
	The Banach space $L_1 (\R)$, when being equipped with the Hartley-Fourier convolution \eqref{eq2.4} multiplication, $\big(L_1 (\R), \underset{H_1, F}{*} \big)$ becomes a non-commutative normed ring and has no unit element.
\end{theorem}
\begin{proof}
	For any $f, g$ which belong to $ L_1 (\R)$ space, we get $(f\underset{H_1, F}{*} g) \in L_1 (\R)$ (refer \cite{ThaoHVA2014MMA}).  Without loss of generality, we may assume that $\|f\|_{L_1(\mathbb{R})} =\sqrt[4]{\frac{2}{\pi}}\int_{\R}|f(x)|dx$ , by inequality \eqref{3.9}, we get $
	\big\|f \underset{H_1,F}{\ast} g\big\|_{L_1(\mathbb{R})} \leq \|f\|_{L_1(\mathbb{R})} \|g\|_{L_1(\mathbb{R})}, \forall f, g\in L_1(\mathbb{R}).
	$	This implies that structure $(L_1 (\R), \underset{H_1, F}{*} )$ is a normed ring. On the other hand, by
	the factorization identity \eqref{eq2.6}, we conclude that the normed ring here is non-commutative
	for the aforementioned convolutional multiplication.
	
	Next, we need to prove that $(L_1 (\R), \underset{H_1, F}{*} )$ has no unit element. Suppose that, there exists  an element $\mathcal{U} \in L_1 (\R)$ such that $(\mathcal{U}\underset{H_1 , F}{*}f)(x)=(f\underset{H_1, F}{*}\mathcal{U})(x)=f(x),$ for any  $f \in L_1 (\R).$ The factorization identity \eqref{eq2.6} implies $
	H_1(f \underset{H_1,F}{\ast} \mathcal{U})(y) = (H_1f)(y), \forall f\in L_1(\mathbb R).
	$ This is equivalent to $(H_1f)(y)(F\mathcal{U})(y)=(H_1f)(y),$ so we deduce that \begin{equation}\label{H1}
	(H_1f)\big[(F\mathcal{U})(y)-1\big] =0, \forall f\in L_1(\mathbb{R}).
	\end{equation} Now, we choose $f(x)=\sqrt{\frac{\pi}{2}} e^{-|x|}\in L_1(\mathbb{R})$.  Obviously $f(x)$ belongs to the $L_1 (\R)$ and we obtain\\
	$
	H_1\left(\sqrt{\frac{\pi}{2}} e^{-|x|}\right) (y)=\frac{1}{1+y^2}>0.
	$  Then
	the equality \eqref{H1} holds if and only if $(F \mathcal{U})(y)=1$, this means that \begin{eqnarray*}
		\mathcal{U} (y)=(F1)(y)=\frac{1}{\sqrt{2\pi}}\int_{\R} e^{-ixy}dx
		=\frac{1}{\sqrt{2\pi}}\int_{\R} (\cos xy - i\sin xy)dx \notin L_1(\mathbb{R}).
	\end{eqnarray*}
	Therefore, the equality \eqref{H1} is not satisfied  $\forall f\in L_1(\R)$, implying that the
	convolution multiplications have no unit element.
\end{proof}
In a similar way, we also obtain a structure non-commutative normed ring, that has no unit element, when $L_1(\R)$ space is equipped with the Hartley convolution \eqref{eq2.5} multiplication.
We next turn our attention to the zero-divisor
properties of the normed ring, which is  Titchmarsh's theorem \cite{Titchmarsh1986,NaimarkMA1972}. The Titchmarsh convolution theorem asserts that if $f, g \in L_1(\mathbb{R})$ vanish on $(-\infty, 0)$, and $(f * g)(x)=0$ for $x \leq T$, then there exist real numbers $\alpha$ and $\beta$, such that $\alpha+\beta=T$, and the functions $f$ and $g$ vanish on $[0, \alpha]$ and $[0, \beta]$, respectively. Based on this, 
we give a new version of the Titchmarsh-type theorem for the Hartley-Fourier generalized convolution operator \eqref{eq2.4}. To get started, we need the following definition.
\begin{definition}
	Denote $ L_1 (\R, e ^{|x|})$ be the normed space of all measurable functions $f(x)$ on $\R$ such that  
	$ \int\limits_{\R} |f(x)|\ e^{|x|} dx < +\infty,$ with the norm defined by $||f||_{L_1 (\R, e^{|x|})} :=\int\limits_{\R} |f(x)|\ e^{|x|} dx.$	
\end{definition}
\begin{theorem}[\textbf{The Titchmarsh-type Theorem}]\label{TitchmarshTheorem}
	Let $f, g$ be continuous functions on $\R$ and such that $f,g \in L_{1}(\mathbb{R}, e^{|x|})$. It follows that $(f \underset{H_1, F}{*}g)(x) \equiv 0$ almost everywhere on $\mathbb{R}$ if and only if $f(x)$ or $g(x)$ is vanish almost everywhere on $\R$. 	
\end{theorem}
\begin{proof}
	If $(f \underset{H_1,F}{\ast} g)\equiv 0$ almost everywhere on $\R$, then $H_1(f \underset{H_1,F}{\ast} g)(y)$ is vanish almost everywhere on $\R$. From \eqref{eq2.6}, we obtain
	\begin{equation}\label{4.2}
	H_1(f \underset{H_1,F}{\ast} g)(y)=(H_1f)(y)(Fg)(y)\equiv 0\ \text{almost everywhere on }\ \R.
	\end{equation}
	We prove that $(H_1 f)(y)$ and $(Fg)(y)$ are analytic functions. Indeed, using \eqref{eq2.1}  together with the Lebesgue’s dominated convergence theorem, we can exchange orders of the differentiation with the
	integration, we infer that
	$$\begin{aligned}
	\left|\frac{d^k}{dy^k}(H_1f)(y)\right| &\leq \frac{1}{2\sqrt{2\pi}}\int_{\R} \left|f(x)\left(\frac{d^k}{dy^k}\mathrm{Cas}\,xy\right)\right|dx
	= \frac{1}{2~\sqrt{2\pi}}\int_{\R} \left|f(x)x^k\right|\left|\cos\left(xy+\frac{k\pi}{2}\right) + \sin\left(xy+\frac{k\pi}{2}\right)\right|dx\\
	&\leq \frac{1}{2~\sqrt{\pi}} \int_{\R} \left|f(x)x^k\right|dx
	\leq \frac{1}{2~\sqrt{\pi}}\int_{\R} e^{-|x|} \frac{|x|^k}{k!} k!|f(x)| e^{|x|} dx\\
	&\leq \frac{1}{2~\sqrt{\pi}}\int_{\R} k!|f(x)| e^{|x|} dx\quad \left(\text{since } e^{-|x|}|x^k| = e^{-|x|} \frac{|x^k|}{k!} k!\leq k!\right)\\
	&=\frac{1}{2~\sqrt{\pi}}k! \|f\|_{L_1(\mathbb R, e^{|x|})} = k!C,\ \text{where }\ C=\frac{1}{2~\sqrt{\pi}}\|f\|_{L_1(\mathbb R, e^{|x|})}.
	\end{aligned}$$
	On the other hand, we evaluate the remainder of Taylor's series expansion in
	 neighborhood of a point $y_0$ as follows
	$$\left|\frac{1}{k!} \frac{d^k(H_1f)(\theta+y_0)}{dy^k} (y-y_0)^k\right| \leq \frac{1}{k!}k! C|y-y_0|^k
	=C |y-y_0|^k,\ \text{with}\ 0<\theta<1.$$
	This means that at $y_0$ and in the neighborhood of $y_0$  $(H_1 f)(y)$ can be expanded to Taylor series and converge on $|y-y_0|<1$, which implies that $(H_1 f)(y)$ is an analytic function $\forall y \in \R$.  In a similar way, we can also prove $(Fg)(y)$ is an analytic function. Therefore, from \eqref{4.2}, we infer that $(H_1 f)(y)$ or $(Fg)(y)$ is vanish almost everywhere on $\R$. Furthermore, the uniqueness in $L_1 (\R)$ of Hartley, Fourier transforms (refer \cite{Bracewell1986,WRudin1987,Sogge1993fourier}), imples that $f(x)$ or $g(x)$ are functions of vanishing almost everywhere on $\R$. 
\end{proof}
Theorem \ref{TitchmarshTheorem} can also be stated for convolution \eqref{eq2.5} with a similar proof.
\subsection{$L_1$-solutions of the Fredholm integral equation second kind}

Now, we consider the Fredholm integral equation of the second kind as follows (see \cite{Srivastava1992Buschman})
\begin{equation}\label{fredholmeq}
f(x)-\lambda \int_{\textbf{S}} K(x,y)f(y)d(y)=g(x), \end{equation}
Here the right-hand side $g( x )$ and the kernel $K(x,y)$ are some known functions, $\lambda$ is a known parameter, $S$ is a piecewise-smooth surface (or line), and $f$ is an unknown function.
Up to now, during the last two decades, the theory of abstract Volterra and Fredholm integral equation has undergone rapid development. To a large extent, this was due to the applications of this theory to problems in mathematical physics, such as viscoelasticity, heat conduction in materials with memory, electrodynamics with memory, and to the need for tools to tackle the problems arising in these fields. Many interesting phenomena are not found with differential equations but observed in specific examples of integral equations (refer \cite{Pruss2013evolutionary,Mezhlum2004Antonio}). For general kernels $K(x,y)$, an explicit solution to Eq. \eqref{fredholmeq} is not known, and approximate solutions have been sought instead. Nevertheless, some authors tried to get explicit analytic solutions to particular cases, for example, in \cite{Srivastava1992Buschman} have found analytic solutions to Eq. \eqref{fredholmeq} for the kernels form $K(x,\tau) = K(x-\tau)= (x-\tau)^\alpha; e^{-a|x-\tau|}; \sinh(a(x-\tau))$, and
$a J_1 (a(x-\tau)$ with $J_1$ being the Bessel functions.
We will consider solvability Eq. \eqref{fredholmeq} in  $L_1 (\R)$ for case $\lambda=-1$, $S= \R$ with kernel $
K(x,y)=\frac{1}{2\sqrt{2\pi}}[\varphi(x+y) +\varphi(x-y) +\varphi(-x+y) -\varphi(-x-y)],
$ and $
g(x)=(\varphi \underset{H_1,F}{\ast} \xi)(x).
$ We can be rewrite \eqref{fredholmeq} in the convolution form as follows
\begin{equation}\label{4.5}
f(x) + \big(f \underset{H_{\left\{\substack{1\\2}\right\}}}{\ast} \varphi\big) (x) =(\varphi\underset{H_1,F}{\ast}\xi)(x), \quad x\in \R.
\end{equation}
\begin{theorem}\label{L1redholm}
	Suppose that $\varphi$ and $\xi$ are given functions belonging to  $L_1(\mathbb R)$  and satisfy the condition $1+(H_1\varphi)(y)\ne 0$ for any $y\in \mathbb R$. Then  \eqref{4.5} has the unique solution in $L_1(\mathbb R)$ which can be
	represented in the form
	$
	f(x) = (\ell \underset{H_1,F}{\ast} \xi)(x).
	$ Here $\ell\in L_1(\mathbb R)$  is defined by
	$
	(H_1\ell)(y)=\frac{(H_1\varphi)(y)}{1+(H_1\varphi)(y)}.
	$
	Furthermore, the following $L_1$-norm estimate holds
	$
	\|f\|_{L_1(\mathbb R)} \leq \sqrt{\frac{2}{\pi}}\|\ell\|_{L_1(\mathbb R)} \|\xi\|_{L_1(\mathbb R)}.
	$		
\end{theorem}
\begin{proof}
	Applying the $(H_1)$ transform to both sides of  \eqref{4.5}, we deduce that
	$$
	(H_1f)(y) + H_1(f \underset{H_{\left\{\substack{1\\2}\right\}}}{\ast} \varphi)(y) = H_1\big(\varphi\underset{H_1,F}{\ast} \xi\big)(y).
	$$
	Using the factorization properties \eqref{eq2.7} and \eqref{eq2.6}, together with the condition $1+(H_1\varphi)(y)\ne 0$ for any $y\in \mathbb R$, we obtain 
	$
	(H_1f)(y)[1+(H_1\varphi)(y)] = (H_1\varphi)(y)(F\xi)(y),\quad y\in\mathbb R
	$
	This is equivalent to 
	$
	(H_1f)(y) = \frac{(H_1\varphi)(y)}{1+ (H_1\varphi)(y)}(F\xi)(y).
	$
	By Lemma \eqref{WienerLevyforHartley}, the existence of a function $\ell \in L_1(\mathbb R)$ such that
	$
	(H_1\ell)(y) = \frac{(H_1\varphi)(y)}{1+(H_1\varphi)(y)}.
	$
	This means that
	$$
	(H_1f)(y) = (H_1\ell)(y)(F\xi)(y) = H_1(\ell\underset{H_1,F}{\ast}\xi)(y).
	$$Therefore, $f(x)=\big(\ell\underset{H_1,F}{\ast}\xi\big)(x)$ almost everywhere for any  $x\in \mathbb R$. Moreover, since $\ell, \xi$ are functions belonging  $L_1(\mathbb R)$, we deduce that $f\in L_1(\mathbb R)$ (see \cite{ThaoHVA2014MMA,Tuan2022MMA}). Applying the inequality \eqref{3.9} we obtain
	norm estimation of the solution on $L_1$ space as follows
	$
	\|f\|_{L_1(\mathbb R)} \leq \sqrt{\frac{2}{\pi}}\|\ell\|_{L_1(\mathbb R)} \|\xi\|_{L_1(\mathbb R)}.
	$
\end{proof}

\begin{remark}\label{remark4.1}
	If $p, q, r>1$ satisfy the condition $\frac{1}{p}+\frac{1}{q}=1+\frac{1}{r}$, applying inequality \eqref{eq3.8}, for any  $f \in L_r(\mathbb R)$, $\ell \in L_p(\mathbb R)$, and $\xi \in L_q(\mathbb R)$, we
	get the following solution estimate
	$
	\|f\|_{L_r(\mathbb R)} \leq \sqrt{\frac{2}{\pi}} \|\ell\|_{L_p(\mathbb R)} \|\xi\|_{L_q(\mathbb R)}.
	$
	
	Furthermore, in case $r=\infty$, by applying \eqref{case=vocung}, we obtain $L_{\infty}$-boundedness of solution for Eq.\eqref{4.5} as follows $\lv f\lv_{L_{\infty} (\R)}\leq \sqrt{\frac{2}{\pi}} \lv\ell\lv_{L_p (\R)} \lv\xi\lv_{L_q (\R)}.$
\end{remark}

\subsection{$L_1$-solutions for Barbashin's equations}
We are concerned with integro-differential equation of the form 
\begin{equation}\label{eqBarbashin}
\frac{\partial f(x,y)}{\partial y} = c(x,y)f(x,y) + \int_a^b K(x,y,t) f(x,t) dt + \varphi(x,y).
\end{equation}
Here $c: J \times[a, b] \rightarrow \mathbb{R}, K: J \times[a, b] \times[a, b] \rightarrow \mathbb{R}$, and mapping $\varphi: J \times[a, b] \rightarrow \mathbb{R}$ are given functions, where $J$ is a bounded or unbounded interval, the function $f$ is unknown. The equation \eqref{eqBarbashin} was first studied  by E.A. Barbashin \cite{Barbashin1957} and his pupils. For this reason, this is nowadays called integro-differential  equation of Barbashin type or simply the Barbashin equation. The  \eqref{eqBarbashin} has been applied to many fields such as mathematical physics, radiation propagation, mathematical biology and transport problems, e.g., for more details refer \cite{Appell2000KalZabre}.  One
of the characteristics of Barbashin equation is that studying solvability of the equation is
heavily dependent on the kernel $K(x,s,\rho)$ of the equation. In many cases, we can reduce equation \eqref{eqBarbashin} to the form of an ordinary differential equation and use the Cauchy integral operator or evolution operator to study it when the kernel does not depend on $x$. In some other cases, we need to use the partial integral operator to study this equation (see \cite{Appell2000KalZabre}).  However, in the general case of $K(x,s,\rho)$ as an arbitrary kernel, the problem of finding a solution for Barbashin equation remains open. In some other cases, we need to use the partial integral operator to study this equation (see \cite{Appell2000KalZabre}).   On the other hand, if we view ${A}$ as the operator defined by ${A}:=\partial\textfractionsolidus \partial y -c(x,y) \mathcal{I}$, where $\mathcal{I}$ is the identity operator, then Eq. \eqref{eqBarbashin}  is written in the following form
$\mathcal{A}f(x,y)=\int\limits_a^b K(x,y,t)f(x,t)dt +\varphi(x,y).$ We will consider solvability Eq. \eqref{eqBarbashin} in  $L_1 (\R)$ for case operator $\mathcal{A}f(x,y)=\left(1-\frac{d^2}{dx^2}\right)\big(f \underset{H_{\left\{\substack{1\\2}\right\}}}{\ast} g\big)(x)$ with kernel $K(x,y,t) = -\frac{1}{2\sqrt{2\pi}}[h(x+y) + h(x-y) + h(-x+y) - h(-x-y)]$, domain $(a,b)=\R$, and $\varphi(x,y)=\Big(h \underset{H_{\left\{\substack{1\\2}\right\}}}{\ast} \xi\Big)(x)$. Then original equation \eqref{eqBarbashin} becomes the linear equation and can be rewritten in the convolution form
\begin{equation}\label{4.7}
\left(1-\frac{d^2}{dx^2}\right)\big(f \underset{H_{\left\{\substack{1\\2}\right\}}}{\ast} g\big) (x) + \big(f \underset{H_{\left\{\substack{1\\2}\right\}}}{\ast} h\big)(x)=\big(h \underset{H_{\left\{\substack{1\\2}\right\}}}{\ast} \xi\big)(x),
\end{equation}
where $g$, $h$, $\xi$ are given functions, convolution $( \underset{H_{\left\{\substack{1\\2}\right\}}}{\ast} )$ is defined by \eqref{eq2.5}, and $f$ is the function to find.
\begin{theorem}\label{TheoremBarbashin}
	Let $\xi \in L_1 (\R)$, $g(x)=\sqrt{\frac{\pi}{2}}e^{-|x|}$ , and $h$  be  functions belonging to $L_1(\mathbb R)$ such that $1+\big(H_{\left\{\substack{1\\2}\right\}}h\big)(y)\ne 0$, $\forall y\in\mathbb R$. Then Eq. \eqref{4.7} has the unique solution $f(x)=\Big(\eta \underset{H_{\left\{\substack{1\\2}\right\}}}{\ast} \xi\Big)(x)\in L_1(\mathbb R)$, where $\eta \in L_1(\mathbb R)$ defined by $
	\big(H_{\left\{\substack{1\\2}\right\}}\eta\big)(y) = \frac{\big(H_{\left\{\substack{1\\2}\right\}} h\big)(y)}{1+\big(H_{\left\{\substack{1\\2}\right\}} h\big)(y)}$ for all $y\in\mathbb R.
	$ Moreover, the estimation of $L_1$-norm is as follows  $
	\|f\|_{L_1(\mathbb R)} \leq \sqrt{\frac{2}{\pi}}\|\eta\|_{L_1(\mathbb R)} \|\xi\|_{L_1(\mathbb R)}.
	$
	
\end{theorem}
\begin{proof}
	It is easy to check $g(x)=\sqrt{\frac{\pi}{2}}e^{-|x|}$ actually satisfies the assumptions of Lemma \ref{lemma4.1}. By using \eqref{42}, we obtain
	$
	H_{\left\{\substack{1\\2}\right\}}\bigg\{\big(1-\frac{d^2}{dx^2}\big)\bigg(f \underset{H_{\left\{\substack{1\\2}\right\}}}{\ast} \sqrt{\frac{\pi}{2}} e^{-|t|}\bigg)(x)\bigg\}(y) = (1+y^2) H_{\left\{\substack{1\\2}\right\}} \big(f \underset{H_{\left\{\substack{1\\2}\right\}}}{\ast} \sqrt{\frac{\pi}{2}} e^{-|t|}\big)(y).
	$  Applying the $H_{\left\{\substack{1\\2}\right\}}$ transform to both sides of \eqref{4.7} combining the factorization equality \eqref{eq2.7}, we get
	$$
	(1+y^2) H_{\left\{\substack{1\\2}\right\}}\left(f \underset{H_{\left\{\substack{1\\2}\right\}}}{\ast} \sqrt{\frac{\pi}{2}} e^{-|t|}\right)(y) + H_{\left\{\substack{1\\2}\right\}}\big(f \underset{H_{\left\{\substack{1\\2}\right\}}}{\ast} h\big)(y)= H_{\left\{\substack{1\\2}\right\}}\big(h \underset{H_{\left\{\substack{1\\2}\right\}}}{\ast}\xi\big)(y),\forall y>0,
	$$ equivalent to $(1+y^2)\big(H_{\left\{\substack{1\\2}\right\}}f\big)(y) H_{\left\{\substack{1\\2}\right\}}\left(\sqrt{\frac{\pi}{2}}e^{-|t|}\right)(y) + \big(H_{\left\{\substack{1\\2}\right\}} f\big)(y)\big(H_{\left\{\substack{1\\2}\right\}} h\big)(y)
	= \big(H_{\left\{\substack{1\\2}\right\}} h\big)(y) \big(H_{\left\{\substack{1\\2}\right\}}\xi\big)(y).$ Since $H_{\left\{\substack{1\\2}\right\}}\Big(\sqrt{\frac{\pi}{2}}e^{-|t|}\Big)(y) = \frac{1}{1+y^2}$ is finite (see \cite{bateman1954}, 1.4.1, page 23), under the condition $1+\big(H_{\left\{\substack{1\\2}\right\}}h\big)(y)\ne 0$, $\forall y\in\mathbb R$, we deduce that $\big(H_{\left\{\substack{1\\2}\right\}} f\big)(y)=\frac{\big(H_{\left\{\substack{1\\2}\right\}}h\big)(y)}{1+\big(H_{\left\{\substack{1\\2}\right\}} h\big)(y)} \big(H_{\left\{\substack{1\\2}\right\}} \xi\big)(y)$.
	Following Lemma \ref{WienerLevyforHartley},  there exists a function $\eta$ which belongs to $L_1(\mathbb R_+)$ such that
	$$
	\big(H_{\left\{\substack{1\\2}\right\}}\eta\big)(y)=\frac{\big(H_{\left\{\substack{1\\2}\right\}} h\big)(y)}{1+\big(H_{\left\{\substack{1\\2}\right\}}h\big)(y)},\quad \forall y\in \mathbb{R}.
	$$
	Thus, we infer that $\big(H_{\left\{\substack{1\\2}\right\}} f\big)(y)=\big(H_{\left\{\substack{1\\2}\right\}}\eta\big)(y)\big(H_{\left\{\substack{1\\2}\right\}}\xi\big)(y)=H_{\left\{\substack{1\\2}\right\}}\big(\eta \underset{H_{\left\{\substack{1\\2}\right\}}}{\ast} \xi\big)(y)$, implies that $
	f(x)=\big(\eta\underset{H_{\left\{\substack{1\\2}\right\}}}{\ast}\xi\big)(x)$ almost everywhere on $\R$. Since $\eta,\xi\in L_1(\mathbb R)$, then $f$ belongs to $L_1(\mathbb R)$ (see \cite{ThaoHVA2014MMA,Tuan2022MMA}).  Applying the inequality \eqref{3.12}, we obtain the boundedness in $L_1$ of solution$
	\|f\|_{L_1(\mathbb R)} \leq \sqrt{\frac{2}{\pi}}\|\eta\|_{L_1(\mathbb R)}\|\xi\|_{L_1(\mathbb R)}.
	$
\end{proof}
\begin{remark}\label{remark4.2}
	Let  $p,q,r>1$ such that $\frac{1}{p}+\frac{1}{q}=1+\frac{1}{r}$ and if $f\in L_r(\mathbb R)$, $\eta\in L_p(\mathbb R)$, and $\xi\in L_q(\mathbb R)$. Applying the inequality \eqref{3.11}, we
	get the following solution estimate
	$
	\|f\|_{L_r(\mathbb R)} \leq \sqrt{\frac{2}{\pi}}\|\eta\|_{L_p(\mathbb R)} \|\xi\|_{L_q(\mathbb R)}.
	$
\end{remark}
\begin{remark}\label{remark3}
If there is a case where $\eta$ is the product of two functions ($\eta=\vp.\rho$), $\rho $ is a  positive function belonging to $L_1 (\R)$, and $\vp\in L_1 (\R,\rho)\cap L_p (\R,\rho)$, with $p>1$. For any function $\xi \in L_1 (\R)\cap L_p (\R)$, apply \eqref{eq3.20}, then we have the following estimate
$\lv f\lv_{L_p (\R)} \leq \frac{1}{\pi}(\sqrt{2\pi})^{\frac{1}{p}} \lv \rho\lv^{1-\frac{1}{p}}_{L_1 (\R)}\lv\vp.\xi\lv_{L_p (\R,\rho)}\lv\xi \lv_{L_p (\R)}$.
\end{remark}
\subsection{The Cauchy-type initial value problem}
We consider the following Cauchy-type problem
\begin{equation}\label{Cauchy-type}
\left\{\begin{array}{l}
f(x)-f''(x)+\big(1-\frac{d^2}{dx^2}\big)\big(f \underset{H_1,F}{\ast}g\big)(x)=\big(h \underset{H_1,F}{\ast} g\big)(x),\\
\lim\limits_{|x|\to \infty} f(x) =\lim\limits_{|x|\to \infty} f'(x) =0,
\end{array}\right.
\end{equation}
where $g$ is a given function belonging to $g\in L_1(\mathbb R)$, $h\in C^2(\mathbb R)\cap L_1(\mathbb R)$ such as $h', h'' \in L_1(\mathbb R)$, and $f$ is the function to find.
\begin{theorem}\label{Theorem45}
	With the assumption of the function h, g us given as above, so that the condition $1+(Fg)(y)\ne 0$ for any $y\in \mathbb R$. Then problem \ref{Cauchy-type}  has a unique solution in $C^2(\mathbb R)\cap L_1(\mathbb R)$ and is represented by \\$
	f(x) = \left(\big(h\underset{H_1,F}{\ast} \xi\big)\underset{H_1,F}{\ast} \sqrt{\frac{\pi}{2}}e^{-|t|}\right)(x). 
	$ Moreover, we have $
	\|f\|_{L_1(\mathbb R)} \leq 2\sqrt{\frac{2}{\pi}}\|h\|_{L_1(\mathbb R)} \|\xi\|_{L_1(\mathbb R)},
	$ where $\xi$ is a
	function beloning to $L_1(\mathbb R)$  determined by
	$
	(F\xi)(y)=\frac{(Fg)(y)}{1+(Fg)(y)},$ and $(\underset{H_1,F}{\ast})$ defined by \eqref{eq2.4}.
\end{theorem}
\begin{proof}
	Without loss of generality, we may assume that $f \in C^2(\mathbb R)\cap L_1(\mathbb R)$ is a satisfied function condition of the Problem \eqref{Cauchy-type}. Applying Lemma \ref{lemma4.1} to the function $f$ instead of the role of $g$, by \eqref{42}, we obtain
	\begin{equation}\label{4.10}
	H_1\left(\left(1-\frac{d^2}{dx^2}\right)(f\underset{H_1,F}{\ast} g)(x)\right)(y)=(1+y^2) H_1(f\underset{H_1,F}{\ast} g)(y),\quad \forall g\in L_1(\mathbb R),~y\in\mathbb R.
	\end{equation}	
	From \eqref{4.3}, we deduce that \begin{equation}\label{4.11}
	\left(H_1f''(x)\right)(y) = -y^2(H_1f)(y).
	\end{equation}	
	Applying the $(H_1)$ transform to both sides of	Eq. \eqref{Cauchy-type}, combining \eqref{4.10}, \eqref{4.11}, and using  the factorization property \eqref{eq2.6}, we infer 
	$
	(H_1f)(y) + y^2(H_1f)(y) + (1+y^2)H_1(f\underset{H_1,F}{\ast} g) (y) = H_1(f\underset{H_1,F}{\ast} g)(y),\forall y\in\mathbb R.
	$ This equivalent to $
	(1+y^2)(H_1f)(y)[1+(Fg)(y)]=(H_1h)(y) (Fg)(y),
	$ implies 
	\begin{equation}\label{4.12}
	(H_1f)(y)=\frac{1}{1+y^2}\frac{(Fg)(y)}{1+(Fg)(y)}(H_1h)(y).
	\end{equation}
	By Wiener-Levy's Theorem \cite{NaimarkMA1972} for the Fourier transform, we deduce  the existence of a function $\xi$ belonging to $L_1(\mathbb R)$ such that
	$
	(F\xi)(y)=\frac{(Fg)(y)}{1+(Fg)(y)}.
	$
	On the other hand, $F\left(\sqrt{\frac{\pi}{2}}e^{-|t|}\right)(y)=\frac{1}{1+y^2}$ \cite{bateman1954} is finite, and  coupling \eqref{4.12} with \eqref{eq2.6}, we have
	$$
	(H_1f)(y) = H_1\left((h\underset{H_1,F}{\ast} \xi) \underset{H_1,F}{\ast} \sqrt{\frac{\pi}{2}}e^{-|t|}\right)(y).
	$$ Therefore $
	f(x) =\left((h\underset{H_1,F}{\ast} \xi)\underset{H_1,F}{\ast} \sqrt{\frac{\pi}{2}} e^{-|t|}\right)(x)$ almost everywhere for any $x\in \R.
	$  Moreover, since $h, \xi,$ and $\sqrt{\frac{\pi}{2}} e^{-|t|}$  are functions belonging to $L_1 (\R)$, by \cite{Tuan2022MMA,ThaoHVA2014MMA} lead to $f\in L_1 (\R)$. We putting $G(x)=(h \underset{H_1,F}{\ast} \xi)(x) \in L_1(\mathbb R)$. Thus $h \in C^2(\mathbb R)\cap L_1(\mathbb R)$, and $h', h'', \xi\in L_1(\mathbb R)$, by \eqref{eq2.4}, we infer $G\in C^2(\mathbb R)\cap L_1(\mathbb R)$. This means that $f\in C^2(\mathbb R)\cap L_1(\mathbb R)$.
	
	Now, we prove solution $f(x)$ of
	Problem \eqref{Cauchy-type} actually satisfies the initial condition. Indeed,  by setting $G(x)$, we obtain
	\begin{equation}\label{4.13}
	\begin{aligned}
	f(x)&=\left(G \underset{H_1,F}{\ast} \sqrt{\frac{\pi}{2}} e^{-|t|}\right)(x) = \frac{\sqrt{\pi}}{2\sqrt{2}}\bigg\{\frac{1}{\sqrt{2\pi}}\int_{\R} e^{-|y|} G(x+y) dy + \frac{1}{\sqrt{2\pi}} \int_{\R} e^{-|y|} G(x-y) dy\\ &+ \frac{i}{\sqrt{2\pi}}\int_{\R} e^{-|y|} G(-x-y)dy - \frac{i}{\sqrt{2\pi}}\int_{\R} e^{-|y|}G(-x+y)dy\bigg\}\\
	&= 2\sqrt{\frac{\pi}{2}}\big\{J_1(x) + J_2(x) + iJ_3(x) - iJ_4(x)\big\},\quad \forall x\in\mathbb R.
	\end{aligned}\end{equation}	
	Observe the integral $J_1$ then
	$
	J_1(x)=\frac{1}{\sqrt{2\pi}} \int\limits_{\R} e^{-|y|} G(x+y)dy =-\frac{1}{\sqrt{2\pi}}\int\limits_{\R}  e^{-|t|}G(x-t)dt=-\big(e^{-|t|}\underset{F}{\ast} G\big)(x).
	$ We apply the Wiener-Tauberian's theorem \cite{Winer1932} to the function $\omega(t)=e^{-|t|}$ with bounded $\forall t\in\mathbb R$, and $\Phi = e^{-|x-y|}\in L_1(\mathbb R)$. It is easy to see $
	F\left(e^{-|x-y|}\right) = \sqrt{\frac{2}{\pi}} \frac{1}{1+(x-y)^2}\ne 0, \forall x, y \in \mathbb R.
	$ On the other hand, 
	\begin{eqnarray*}
		\big(\omega\underset{F}{\ast} \Phi\big)(x) = \frac{1}{\sqrt{2\pi}}\int_{\R} e^{(-|y|-|x-y|)} dy= \left\{\begin{array}{ll}
			\frac{1}{\sqrt{2\pi}}(x+1)e^{-x}&\quad \text{if}\ x>0,\\
			\frac{1}{\sqrt{2\pi}}(1-x)e^x&\quad \text{if}\ x<0,
		\end{array}\right.\quad\text{tends to}\ 0\ \text{when }|x|\to\infty.
	\end{eqnarray*}
	By Wiener-Tauberian's theorem, we deduce that $
	\lim\limits_{|x|\to \infty} J_1(x) = - \lim\limits_{|x|\to \infty}\big(e^{-|t|}\underset{F}{\ast} G\big)(x) =0, \forall G\in L_1(\mathbb R).
	$ By the same argument as above, we obtain $
	\lim\limits_{|x|\to \infty}J_i(x)=0,$ with $i=2,3,4. 
	$ From \eqref{4.13} we get $
	\lim_{|x|\to \infty}f(x)=0.
	$ Furthermore,  we also deduce $
	f'(x)=2\sqrt{\frac{\pi}{2}}\left\{J'_1(x) - J'_2(x) - iJ'_3(x) + iJ'_4(x)\right\},
	$ where for each integral $J'_i(x)$, with $i=\overline{1,4}$ can be represented by the Fourier convolution of the form $\big(e^{-|t|}\underset{F}{\ast} G'\big)$, $\forall G'\in L_1(\mathbb R)$. Hence, we apply the Wiener-Tauberian's theorem again \cite{Winer1932}(also refer \cite{WRudin1987}) to the function $G'$, with $\omega$, $\Phi$ as above  implies that $
	\lim\limits_{|x|\to\infty} J'_i(x) = 0,
	$ with $i=\overline{1,4}$. This yields  $\lim\limits_{|x|\to\infty}f'(x)=0$, and we can conclude that  $f(x)$ satisfies the initial conditions of the Problem \eqref{Cauchy-type}. Applying the equality \eqref{3.9}, we obtain the boundedness in $L_1$ of solution as follows
	$$\begin{aligned}
	\|f\|_{L_1(\mathbb R)} &=\big\|(h \underset{H_1,F}{\ast}\xi) \underset{H_1,F}{\ast} \sqrt{\frac{\pi}{2}}e^{-|t|}\big\|_{L_1(\mathbb R)}\leq  \sqrt{\frac{2}{\pi}}\big\|h\underset{H_1,F}{\ast} \xi\big\|_{L_1(\mathbb R)} \big\|\sqrt{\frac{\pi}{2}}e^{-|t|}\big\|_{L_1(\mathbb R)}\\
	&\leq \sqrt{\frac{2}{\pi}}\|h\|_{L_1(\mathbb R)} \|\xi\|_{L_1(\mathbb R)} \|e^{-|t|}\|_{L_1(\mathbb R)}=2\sqrt{\frac{2}{\pi}}\|h\|_{L_1(\mathbb R)}\|\xi\|_{L_1(\mathbb R)}.
	\end{aligned}$$
\end{proof}
\begin{remark}\label{remmark4.3}
	Assume that $p, q, r, s>1$ and satisfy the condition $\frac{1}{p}+\frac{1}{q}+\frac{1}{r}=2+\frac{1}{s}$. If $f\in L_s(\mathbb R)$, $h\in L_p(\mathbb R)$, and $\xi\in L_q(\mathbb R)$, based on \eqref{eq3.8}, we
	get the following solution of Problem \eqref{Cauchy-type} estimate
	$$
	\|f\|_{L_s(\mathbb R)} \leq \sqrt{\frac{2}{\pi}}\left(\frac{2}{r}\right)^{\frac{1}{r}}\|h\|_{L_p(\mathbb R)}\|\xi\|_{L_q(\mathbb R)}.
	$$
\end{remark}
\section{Examples}
\noindent In the last section of article, we give numerical examples to illustrate the calculation to ensure the validity and applicability of the results in Section 4.
First, we get started with an illustrative example of Theorem \ref{L1redholm} and Remark \ref{remark4.1}.
\begin{example}\label{vd1}
	Let 
	$
	\varphi(x)=\xi(x)=\sqrt{\frac{\pi}{2}}e^{-|x|}$. It is easy to check $\varphi(x)$ and $\xi$ are functions belonging to  $ L_1(\mathbb R).$ According to formula 1.4.1, p. 23 in \cite{bateman1954}, we have
	$
	H_1\left(\sqrt{\frac{\pi}{2}}e^{-|x|} \right)(y) =\frac{1}{1+y^2}.
	$\\ Therefore $1+(H_1\varphi)(y)\ne 0$, $\forall y\in\mathbb R$ and 
	$
	\frac{(H_1\varphi)(y)}{1+(H_1\varphi)(y)}=\frac{1}{2+y^2} = (H_1\ell)(y),
	$ by Lemma \ref{WienerLevyforHartley}, we infer there existence of a function $\ell \in L_1 (\R)$ is satisfied. In this case $\ell(x)=\frac{\sqrt{\pi}}{2}e^{-\sqrt{2}|x|}\in L_1(\mathbb R)$. We obtain the solution in $L_1 (\R)$ of Eq. \eqref{4.5} as follows $
	f(x)=\left(\frac{\sqrt{\pi}}{2}e^{-\sqrt{2}|t|} \underset{H_1,F}{\ast} \sqrt{\frac{\pi}{2}}e^{-|t|}\right)(x).
	$ We get the $L_1$-norm estimation of $f(x)$ as follows $
	\|f\|_{L_1(\mathbb R)} \leq \sqrt{\frac{2}{\pi}}\big\|\frac{\sqrt{\pi}}{2}e^{-\sqrt{2}|t|}\big\|_{L_1(\mathbb R)} \big\|\sqrt{\frac{\pi}{2}}e^{-|t|}\big\|_{L_1(\mathbb R)} =\sqrt{2\pi}.
	$
	Moreover, if there exist real numbers $p, q, r>1$ and satisfy $\frac{1}{p}+\frac{1}{q}=1+\frac{1}{r}$, then
$$
		\|f\|_{L_r(\mathbb R)}\leq  \sqrt{\frac{2}{\pi}}\big\|\frac{\sqrt{\pi}}{2}e^ {-\sqrt{2}|t|} \big\|_{L_p(\mathbb R)}\big\|\sqrt{\frac{\pi}{2}}e^{-|t|} \big\|_{L_q(\mathbb R)}= \frac{\sqrt{\pi}}{2}\left(\frac{\sqrt{2}}{p}\right)^{\frac{1}{p}}\left(\frac {2}{q}\right)^{\frac{1}{q}}.
$$
For case $r=\infty$, from inequality \eqref{case=vocung}, we obtain $\|f\|_{L_\infty(\mathbb R)}\leq \frac{\sqrt{\pi}}{2}\left(\frac{\sqrt{2}}{p}\right)^{\frac{1}{p}}\left(\frac {2}{q}\right)^{\frac{1}{q}}.$
\end{example}
Next example illustrates the Theorem \ref{TheoremBarbashin}.
\begin{example}
	Choose functions $
	h(x)=\xi(x)=\sqrt{\frac{\pi}{2}} e^{-|x|}. 
	$ Obviously, these functions are belonging to $L_1 (\R)$.  According to formula 1.4.1, p. 23 in \cite{bateman1954}, we deduce that $
	H_{\left\{\substack{1\\2}\right\}}\left(\sqrt{\frac{\pi}{2}}e^{-|x|}\right)(y)=\frac{1}{1+y^2}.
	$\\ Thus $1+H_{\left\{\substack{1\\2}\right\}}\left(\sqrt{\frac{\pi}{2}}e^{-|x|}\right) \neq 0, \forall y \in \R$ and $
	\frac{(H_{\left\{\substack{1\\2}\right\}}h)(y)}{1+(H_{\left\{\substack{1\\2}\right\}}h)(y)}=\frac{1}{2+y^2} = (H_{\left\{\substack{1\\2}\right\}}\eta)(y),
	$ by Lemma \ref{WienerLevyforHartley}, we infer there existence of a function $\eta=\frac{\sqrt{\pi}}{2}e^{-\sqrt{2}|x|} \in L_1 (\R)$ is satisfied. Therefore, Eq. \eqref{4.7} has the unique solution $f(x)=\bigg(\frac{\sqrt{\pi}}{2}e^{-\sqrt{2}|t|}\underset{H_{\left\{\substack{1\\2}\right\}}}{*} \sqrt{\frac{\pi}{2}}e^{-|t|}\bigg)(x)$ belonging to $L_1 (\R).$ The estimate of $f(x)$ is similar in Example \ref{vd1}  means $\|f\|_{L_1 (\R)} \leq \sqrt{2\pi}$ and $\|f\|_{L_r (\R)} \leq \frac{\sqrt{\pi}}{2}\left(\frac{\sqrt{2}}{p}\right)^{\frac{1}{p}}\left(\frac {2}{q}\right)^{\frac{1}{q}}.$
	
\noindent	Now we rewrite function $\eta$ in the following form: $$\eta=\frac{\sqrt{\pi}}{2}e^{-\sqrt{2}|t|}=\frac{\sqrt{\pi}}{2}e^{-\frac{\sqrt{2}}{2}|t|}.e^{-\frac{\sqrt{2}}{2}|t|}.$$
	We set $\vp=\frac{\sqrt{\pi}}{2}e^{-\frac{\sqrt{2}}{2}|t|}$ and weighted function $\rho=e^{-\frac{\sqrt{2}}{2}|t|}$. Clearly, in this case, $\vp \in L_1 (\R,\rho)\cap L_p (\R,\rho)$ and $\rho$ is a positive function belonging to $L_1 (\R)$. According to Remark \ref{remark3}, we obtain
	$$\begin{aligned}
	\lv f\lv_{L_p (\R)} &\leq \frac{1}{\pi}(\sqrt{2\pi})^{\frac{1}{p}}\lv e^{-\frac{\sqrt{2}}{2}|t|}\lv_{L_1 (\R)}^{1-\frac{1}{p}} \lv\frac{\sqrt{\pi}}{2}e^{-\frac{\sqrt{2}}{2}|t|}\lv_{L_p (\R,\rho)} \lv \frac{\sqrt{\pi}}{2}e^{-\frac{\sqrt{2}}{2}|t|}\lv_{L_p (\R)},\ \forall p>1.\\&
	\leq \sqrt{2}\bigg(\frac{2\sqrt{2\pi}}{p(p+1)} \bigg)^{\frac{1}{p}}.
	\end{aligned}$$
\end{example}
We end the article with an example for the Theorem \ref{Theorem45}.
\begin{example}
A simple example is provided by  $g=h=\sqrt{\frac{\pi}{2}}e^{-|x|}$. It is not difficult to check that  $g$ belongs to $L_1(\mathbb R)$, and $h\in C^2(\mathbb R)\cap L_1(\mathbb R)$. Obviously  $h', h'' \in L_1(\mathbb R)$, following \cite{bateman1954}, we have $F\big(\sqrt{\frac{\pi}{2}}e^{-|x|}\big)(y)=\frac{1}{1+y^2}$, therefore $1+(Fg)(y)\ne 0$, $\forall y\in\mathbb R$. Moreover
$
(F\xi)(y)=\frac{(Fg)(y)}{1+(Fg)(y)}=\frac{1}{2+y^2},
$ so there is the function $\xi(x)=\frac{\sqrt{\pi}}{2}e^{-\sqrt{2}|x|} \in L_1(\mathbb R)$. In this case, the solution of Problem \eqref{Cauchy-type} is $$
f(x)=\left(\left(\sqrt{\frac{\pi}{2}}e^{-|t|}\underset{H_1,F}{\ast} \frac{\sqrt{\pi}}{2}e^{-\sqrt{2}|t|}\right) \underset{H_1,F}{\ast} \sqrt{\frac{\pi}{2}}e^{-|t|}\right)(x),
$$
and boundedness in $L_1$ is as follows $
\|f\|_{L_1(\mathbb R)} \leq 2\sqrt{\frac{2}{\pi}}\big\|\frac{\sqrt{\pi}}{2}e^{-\sqrt{2}|t|}\big\|_{L_1(\mathbb R)} \big\|\sqrt{\frac{\pi}{2}}e^{-|t|}\big\|_{L_1(\mathbb R)}=2\sqrt{2\pi}.
$\\ Furthermore, if we assume that $p, q, r, s>1$ satisfy condition $\frac{1}{p}+\frac{1}{q}+\frac{1}{r}=2+\frac{1}{s}$, we obtain
$$
	\|f\|_{L_s(\mathbb R)}\leq  \sqrt{\frac{2}{\pi}}\left(\frac{2}{r}\right)^{\frac{1}{r}}\left\|\sqrt{\frac{\pi}{2}}e^{-|t|}\right\|_{L_p(\mathbb R)} \left\|\frac{\sqrt{\pi}}{2}e^{-\sqrt{2}|t|}\right\|_{L_q(\mathbb R)}
	=\frac{\sqrt{\pi}}{2}\left(\frac{2}{r}\right)^\frac{1}{r}\left(\frac{2}{p}\right)^\frac{1}{p}\left(\frac{\sqrt{2}}{q}\right)^\frac{1}{q}.$$
\end{example}
\vskip 0.4cm

\noindent{\textbf{Acknowledgements.}} 
The author would like to thank the referee for carefully reading the manuscript and giving valuable comments to improve the article.
\vskip 0.4cm
\noindent \textbf{DATA AVAILIBILITY} \\
My manuscript has no associated data.

\vskip 0.4cm
\noindent \textbf{FUNDING STATEMENT}\\
\noindent No funding was received for conducting this study.

\vskip 0.4cm
\noindent \textbf{CONFLICT OF INTEREST}\\
\noindent The author has no conflicts of interest to declare that are relevant to the content of this article.

\vskip 0.4cm
\noindent \textbf{ORCID}\\
\noindent \textit{Trinh Tuan} {\color{blue}\url{ https://orcid.org/0000-0002-0376-0238}}

\bibliographystyle{plain}
\bibliography{referecesTLTK}
\nocite{hoangtuan2017thaovkt}
\nocite{Winer1932}
\nocite{AdamsFournier2003sobolev}
\nocite{Appell2000KalZabre}
\nocite{Barbashin1957}
\nocite{NaimarkMA1972}
\nocite{Sneddon1972}
\nocite{Sogge1993fourier}
\nocite{Bracewell1986}
\nocite{Titchmarsh1986}
\nocite{Tuan2022MMA}
\nocite{tuan2020Ukrainian}
\nocite{bateman1954}
\nocite{YoungWH1912}
\nocite{WRudin1987}
\nocite{Pruss2013evolutionary}
\nocite{ThaoHVA2014MMA}
\nocite{Poularikas1996}
\nocite{tuan2021korean}
\nocite{Srivastava1992Buschman}
\nocite{Mezhlum2004Antonio}
\nocite{Stein1971Weiss}
\nocite{Debnath2006Bhatta}
\nocite{castro2013}
\nocite{giang2009MauTuan}
\nocite{castro2019NMTuan}
\nocite{Tuan2023VKT}
\nocite{tuan2022Mediterranean}
\nocite{Saitoh2000}
\end{document}